\documentclass[a4paper]{amsart}

\usepackage{subfiles}
\usepackage{hyperref}
\usepackage{mathrsfs}
\usepackage{pdflscape}
\usepackage[export]{adjustbox}
\usepackage{tikz-cd}
\usepackage{todonotes}

\hypersetup{urlcolor=blue, citecolor=blue, linkcolor=blue, colorlinks=true}

\usepackage{enumitem}
\usepackage{array}
\usepackage{amssymb}
\usepackage{mathtools}
\usepackage{marvosym}
\usepackage{wrapfig}
\usepackage{faktor}
\usepackage{amsfonts}
\usepackage{microtype}
\usepackage{multicol}
\usepackage{subcaption}
\usepackage{xcolor}
\usepackage{stackengine,scalerel}
\usepackage{tikz}
\usetikzlibrary{arrows}
\usetikzlibrary{decorations.pathmorphing}
\tikzset{snake it/.style={decorate, decoration=snake}}
\usepackage{tikz-cd}
\usepackage{amsthm}
\usepackage{thmtools}
\usepackage[nameinlink,capitalise]{cleveref}
\usepackage[breakable]{tcolorbox}
\usepackage[utf8]{inputenc}
\usepackage{csquotes}
\usepackage[english]{babel}

\newcolumntype{L}{>{\scriptstyle}l}
\newcolumntype{C}{>{\scriptstyle}c}
\newcolumntype{R}{>{\scriptstyle}r}

  \definecolor{grau1}{RGB}{100, 100, 100}
  \definecolor{grau2}{RGB}{150, 150, 150}
  \definecolor{grau3}{RGB}{200, 200, 200}

\usepackage{newunicodechar}
\usepackage[style=alphabetic,backend=biber,maxnames=5,maxalphanames=5,doi=true,isbn=false,url=false]{biblatex}
\bibliography{literature.bib}
  
\renewbibmacro{in:}{}
\renewbibmacro*{issue+date}{%
  \ifboolexpr{not test {\iffieldundef{year}} or not test {\iffieldundef{issue}}}
    {\printtext[parens]{%
       \iffieldundef{issue}
         {\usebibmacro{date}}
         {\printfield{issue}%
          \setunit*{\addspace}%
          \usebibmacro{date}}}}
    {}%
  \newunit}

\let\oldtocsection=\tocsection
\let\oldtocsubsection=\tocsubsection
\let\oldtocsubsubsection=\tocsubsubsection
\renewcommand{\tocsection}[2]{\hspace{0em}\oldtocsection{#1}{#2}}
\renewcommand{\tocsubsection}[2]{\hspace{1em}\oldtocsubsection{#1}{#2}}
\renewcommand{\tocsubsubsection}[2]{\hspace{2em}\oldtocsubsubsection{#1}{#2}}
\newtheorem{bigthm}{Theorem}

\newtheorem{bigquestion}[bigthm]{Question}
\newtheorem{theorem}{Theorem}[section]
\newtheorem{lemma}[theorem]{Lemma}
\newtheorem{prop}[theorem]{Proposition}
\newtheorem{cor}[theorem]{Corollary}

\newtheorem*{conjecture*}{Conjecture}

\newtheorem{dcj}{Doubling conjecture}
\newenvironment{doublingconjecture}{%
  \dcj}{\enddcj}

\theoremstyle{definition}
\newtheorem{defn}[theorem]{Definition}

\theoremstyle{remark}
\newtheorem{example}[theorem]{Example}

\newtheorem{remark}[theorem]{Remark}

\newcommand{\nocontentsline}[3]{}
\newcommand{\tocless}[2]{\bgroup\let\addcontentsline=\nocontentsline#1{#2}\egroup}

\setenumerate[1]{leftmargin=*,labelindent=0pt,label=(\roman*),}

\newcommand{\pref}[2]{\hyperref[#1]{#2 \ref*{#1}}}
\newcommand{\id}{\ensuremath{\operatorname{id}}}
\newcommand{\im}{\ensuremath{\operatorname{im}}}
\newcommand{\Spin}{{\ensuremath{\operatorname{Spin}}}}
\newcommand{\so}{{\ensuremath{\operatorname{SO}}}}
\newcommand{\ort}{{\ensuremath{\operatorname{O}}}}

\newcommand{\bo}{B\ort}
\newcommand{\bso}{B\so}
\newcommand{\bspin}{B\Spin}

\newcommand{\width}{\ensuremath{\operatorname{width}}}
\newcommand{\injrad}{\ensuremath{\operatorname{injrad}}}
\newcommand{\vol}{\ensuremath{\operatorname{vol}}}
\newcommand{\area}{\ensuremath{\operatorname{area}}}

\newcommand{\scal}{\ensuremath{\operatorname{scal}}}
\newcommand{\tr}{\ensuremath{\operatorname{tr}}}
\newcommand{\sign}{\ensuremath{\operatorname{sign}}}

\newcommand{\colim}{\ensuremath{\operatorname{colim}}}

\newcommand{\too}{\longrightarrow}

\newcommand{\double}{\mathsf{d}}

\newcommand{\embeds}{\hookrightarrow}

\newcommand{\dbyd}[1]{\frac{\mathrm{d}}{\mathrm{d} #1}}

\DeclarePairedDelimiter{\scpr}{\langle}{\rangle}

\newcommand{\bbN}{\mathbb{N}}

\newcommand{\bbZ}{\mathbb{Z}}

\newcommand{\bbR}{\mathbb{R}}

\newcommand{\ahat}{\hat{\mathcal{A}}}

\newcommand{\cp}[1]{\mathbb{CP}^{#1}}
\newcommand{\bcp}[1]{\overline{\mathbb{CP}}^{#1}}

\newcommand{\op}{{\mathrm{op}}}

\newcommand{\ko}{{\mathrm{KO}}}

\newcommand{\cone}{\mathrm{Cone}}

\newcommand{\homology}{\mathrm{H}}
\newcommand{\cohomology}{\mathrm{H}}

\newcommand{\dt}{\mathrm{dt}}

\newcommand{\drho}{\mathrm{d}\rho}
\newcommand{\eucl}{\mathrm{eucl}}
\newcommand{\tor}{\mathrm{tor}}

\newcommand{\justification}[1]{
}

\begin{document}

\author[Georg Frenck]{Georg Frenck}
\email{\href{mailto:georg.frenck@math.uni-augsburg.de}{georg.frenck@math.uni-augsburg.de}}
\email{\href{mailto:math@frenck.net}{math@frenck.net}}
\urladdr{\href{http://frenck.net/Math}{Frenck.net/Math}}
\address{Universität Augsburg, Universitätsstr.~14, 86159 Augsburg, Germany}

\subjclass[2010]{}
\keywords{}

\title[The Doubling conjecture]{The doubling conjecture for positive scalar curvature}

\begin{abstract}
  The doubling conjecture predicts that a manifold admits positive scalar curvature with mean convex boundary if and only if its double admits positive scalar curvature.
  We show that it holds true for manifolds where the inclusion of the boundary satisfies a certain split-condition on fundamental groups.
  Our proof is based on surgery-techniques for positive scalar and mean curvature.
  If the boundary is non-connected, we use existence of area-minimizing hypersurfaces and the monotonicity-formula.
  Furthermore, we investigate if a psc-metric on a closed manifold can be adjusted so that a given embedded hypersurface is minimal, stable minimal or totally geodesic.
  While not true in general, such an adjustment is possible in many cases.
\end{abstract}

\maketitle

\tableofcontents

\section{Introduction}\label{sec:intro}
A classical construction in scalar curvature geometry shows that the double of a positive scalar curvature (psc) metric with mean convex boundary can be deformed to yield a smooth psc-metric\footnote{see \cite[Theorem 5.7]{gromovlawson_spin-and-scalar-curvature-in-the-presence-of-a-fundamental-group} or \cite[Corollary 4.3]{Baer-Hanke_boundary-conditions-for-scalar-curvature}.}.
Motivated by this observation Rosenberg--Weinberger formulate the following conjecture, which predicts that the converse is true, as well.

\begin{doublingconjecture}[{\cite[Conjecture 7.1]{rosenberg-weinberger_positive-scalar-curvature-on-manifolds-with-boundary-and-their-doubles}}]\label{doubling-conjecture}
  A manifold $M$ with boundary admits a metric of positive scalar curvature and mean convex boundary if and only if its double $\double M =M\cup_{\partial M} M^\op$  admits a metric of positive scalar curvature.
\end{doublingconjecture}

In this paper, we show that the \cref{doubling-conjecture} can be controlled by low-dimensional tangential information of $M$, together with a mild group-theoretic splitting condition on fundamental groups.
To state our main result let $M$ be a compact, oriented manifold with boundary $\partial M$, consisting of components $N_1,\dots,N_n$ and denote by $\iota_i\colon N_i\subset\partial M\embeds M$ the respective inclusions of boundary components.
For $i\in\{1,\dots,n\}$ consider the following injective homomorphism induced by $\iota_i$:
\begin{equation}\label{eq:underlying-criterion}
  \begin{tikzcd}
    {\faktor{\pi_1(N_i)\ }{\ \ker(\iota_i)_\ast}}\ar[r, hook] & \pi_1(M).
  \end{tikzcd}
\end{equation}
We split our main result into two cases: First, the spin and totally nonspin case and second, the almost spin case\footnote{See \cref{def:almost-spin-totally-nonspin} for the definition of totally nonspin and almost spin manifolds}.
Both of these are derived from more general results involving the language of tangential structures.

\begin{bigthm}\label{main:doubling-conjecture-spin-and-nonspin}
  Let $M$ be a manifold of dimension at least $5$ which is either spin or totally nonspin.
  Assume that the injective homomorphism from \eqref{eq:underlying-criterion} is split-injective for all $i=1,\dots,n$ and that one of the following is satisfied:
  \begin{enumerate}
    \item $\dim(M)\le 11$ or
    \item for $i\ge2$ the homomorphism $(\iota_i)_\ast\colon \pi_1(N_i)\to\pi_1(M)$ is trivial and, if $M$ is spin, then the $\alpha$-invariant $\alpha(N_i)\in\ko^{d-1}(\ast)$ vanishes.
  \end{enumerate} 
  Then the \cref{doubling-conjecture} holds for $M$.
\end{bigthm}

\begin{remark}
  \begin{enumerate}
    \item 
    If $\partial M$ is connected, (ii) is vacuously satisfied.
    In this case there is no further assumption other than \eqref{eq:underlying-criterion} being split-injective. 
    \item
    Since $(K3\#\bcp{2})\times S^1$ admits a psc-metric, \cref{main:doubling-conjecture-spin-and-nonspin} implies that the totally nonspin manifold $(K3\#\bcp2)\times[-1,1]$ admits a psc-metric with mean convex boundary.
    This was hinted at in \cite[Section 7, (2)]{rosenberg-weinberger_positive-scalar-curvature-on-manifolds-with-boundary-and-their-doubles} as a potential counterexample to the \cref{doubling-conjecture}.
  \end{enumerate}
\end{remark}

\medskip

\noindent In the case of almost spin manifolds, we need an additional assumption.

\begin{bigthm}\label{main:doubling-conjecture-almost-spin}
  Let $M$ be an almost spin manifold of dimension $d\ge5$.
  Assume that the injective homomorphism from \eqref{eq:underlying-criterion} is split-injective and that the inclusion $\iota_i$ induces an isomorphism 
  \begin{align}\label{eq:almost-spin-extra-assumption}
    H^2\bigl(B\pi_1(M);\bbZ/2\bigr)\to H^2\bigl(B\left(\faktor{\pi_1(N_i)\ }{\ \ker(\iota_i)_\ast}\right);\bbZ/2\bigr)
  \end{align}
  for every $i$. 
  If $\partial M$ is not connected, assume further that $\dim(M)\le 11$.
  Then the \cref{doubling-conjecture} holds for $M$.
\end{bigthm}

The assumption on \eqref{eq:almost-spin-extra-assumption} from \cref{main:doubling-conjecture-almost-spin} forces the map from \eqref{eq:underlying-criterion} to be non-trivial and is satisfied for example, if the homomorphism $(\iota_i)_\ast$ from \eqref{eq:underlying-criterion} is an isomorphism for every $i\ge1$.

\begin{remark}
  Before formulating the \cref{doubling-conjecture} in \cite{rosenberg-weinberger_positive-scalar-curvature-on-manifolds-with-boundary-and-their-doubles}, Rosenberg--Weinberger gather some evidence based on obstruction and existence results of Stolz and Führing \cite{stolz_simplyconnected,fuehring_a-smooth-variation-of-baas-sullivan-theory-and-positive-scalar-curvature}.
  As a consequence, their results require $\dim(M)\ge6$ and strong assumptions on the fundamental group of $M$.
  For example, they show that the doubling conjecture holds true for a high-dimensional spin manifold $M$, if its fundamental group has homological dimension at most $\dim(M)-2$ and satisfies both the Baum--Connes conjecture with coefficients and the Gromov--Lawson--Rosenberg conjecture, both of which are known to have counterexamples \cite{Schick_A-counterexample-to-the-unstable-GromovLawsonRosenberg-conjecture,Higson-Lafforgue-Skandalis_Counterexamples-to-the-BaumConnes-conjecture}.

  In contrast, our proofs are based on surgery constructions for positive scalar and mean curvature and on the existence of smooth, area-minimizing hypersurfaces.
  Besides the split condition on the homomorphism \eqref{eq:underlying-criterion}, there are no assumptions on the fundamental group of $M$ or $\partial M$.
  For example, \cref{main:doubling-conjecture-spin-and-nonspin} implies that the doubling conjecture holds for all totally non-spin manifolds and for spin-manifolds up to dimension $11$ for which the inclusions $\iota_i\colon N_i\subset\partial M\embeds M$ of boundary components induce trivial homomorphisms on fundamental groups.
\end{remark}

For $1$-manifolds, the \cref{doubling-conjecture} is vacuously true and in dimension $2$, the doubling conjecture follows from the Gauß--Bonnet theorem.
In dimension $3$ it can be deduced from \cite[Theorem 2.1]{carlotto-li_constrained-deformations-of-positive-scalar-curvature-I}, see \cref{sec:low-dimensional}.

\medskip

In the realm of $4$-manifolds, counterexamples to positive scalar curvature conjectures are often detected by Seiberg--Witten invariants.
For example $K3\#\bcp2$ is a counterexample to \cite[Conjecture 7.1]{Rosenberg-Stolz_Manifolds-of-positive-scalar-curvature} and to the \enquote{$S^1$-stability conjecture} \cite[Conjecture 1.24]{Rosenberg_Manifolds-of-positive-scalar-curvature-a-progress-report}, whereas $4\cp2\#21\bcp2$ carries counterexamples to the \enquote{concordance-implies-isotopy-conjecture} \cite[Problem 6.3]{Rosenberg-Stolz_Metrics-of-positive-scalar-curvature-and-connections-with-surgery}, see \cite[Corollary 5.2]{Ruberman_Positive-scalar-curvature-diffeomorphisms-and-the-SeibergWitten-invariants}.
We can adjust some of our surgery techniques to obtain the following result, which excludes some simple candidates for counterexamples to the $4$-dimensional \cref{doubling-conjecture}.
\begin{bigthm}\label{main:dimension-4}
  \begin{enumerate}
    \item If the \cref{doubling-conjecture} holds for all $4$-manifolds $M$ satisfying $\partial M \cong S^3$, then the \cref{doubling-conjecture} holds for all $4$-manifolds $N$ with $\partial N \cong S^3\amalg \dots\amalg S^3$, too.
    \item If $M$ is given as the connected sum of any number of copies of $\pm K3$ and $4$-manifolds admitting psc-metrics (for example $\cp{2}$, $\bcp{2}$, $S^2\times S^2$ or $S^1\times S^3$), then $M\setminus D^4$ admits a psc-metric with mean convex boundary
    \item Assume that one of the following holds:
    \begin{enumerate}
      \item 
      $M$ is spin and $\pi_1(M)$ is a free group or
      \item
      $M$ is oriented, totally nonspin and $\pi_1(M)$ has homological dimension at most $3$, 
    \end{enumerate}
    then $k(S^2\times S^2)\#M\setminus D^4$ admits a psc-metric with mean convex boundary for some $k\ge0$.
  \end{enumerate}
\end{bigthm}

As a consequence of \cref{main:dimension-4} (i) and (ii), we observe that the $K3$-surface with any number of balls removed satisfies the doubling conjecture.
Combining (i) and (iii) of \cref{main:dimension-4} we furthermore obtain that a $4$-manifold $M$ whose boundary is a disjoint union of $3$-spheres and which satisfies either condition from (iii) \emph{stably} satisfies the doubling conjecture.

\medskip

Beyond existence questions, doubling constructions naturally give rise to distinguished hypersurfaces. Indeed, whenever a smooth metric on a closed manifold arises by doubling a metric on a manifold with boundary, the hypersurface along which the doubling occurs is necessarily minimal. This simple observation leads to a natural geometric problem:

\begin{bigquestion}\label{que:minimal-hypersurface}
  Given a closed, oriented manifold $M$ and a two-sided hypersurface $\Sigma\subset M$, does $M$ admit a psc-metric such that $\Sigma$ is minimal, stable minimal or totally geodesic?
\end{bigquestion}

We show that \cref{que:minimal-hypersurface} has a negative answer in general (see \cref{ex:not-stable-minimal} and \cref{ex:not-minimal}), even for manifolds admitting positive scalar curvature, and we provide explicit counterexamples demonstrating that additional hypotheses are unavoidable.
On the positive side, we identify a broad class of situations in which a given hypersurface can be realized as minimal in a background metric of positive scalar curvature, namely when the hypersurface satisfies a natural extension condition relative to the ambient manifold.

\medskip

Moreover, once a background psc-metric is fixed, \cref{que:minimal-hypersurface} admits a refined, relative formulation: can the metric be modified locally near $\Sigma$, so that $\Sigma$ becomes minimal? 
The following theorem provides criteria for an affirmative answer to both the absolute and the relative version of \cref{que:minimal-hypersurface}.

\begin{bigthm}\label{main:minimal-surfaces}
  Let $M$ be a closed, oriented manifold and $g$ a psc-metric on $M$.
  Let $\Sigma\subset M$ be a two-sided, connected hypersurface.
  Assume that the inclusion $\Sigma\embeds M$ induces split-injection on fundamental groups and that either $M$ is spin or $\Sigma$ is totally nonspin.
  If $\Sigma$ is non-separating, assume further that $\dim(M)\le 11$.
  Then the following hold true: 
  \begin{enumerate}
    \item If $\dim(M)\ge5$, then there exists a psc-metric $\widetilde{g}$ on $M$ which agrees with $g$ outside a tubular neighborhood of $\Sigma$ and $\Sigma$ is minimal with respect to $\widetilde{g}$.
    \item If $\dim (M)\ge6$, then there exists a psc-metric $\widetilde{g}$ on $M$ such that $\widetilde g = \dt^2 + h$ in a neighborhood of $\Sigma$. 
    If $\Sigma$ is separating, $\widetilde g$ can be chosen to agree with $\Sigma$ outside a tubular neighborhood of $\Sigma$.
  \end{enumerate}
\end{bigthm}

\begin{remark}
  \begin{enumerate}
    \item The restriction to $\dim(M)\ge6$ in part $(ii)$ of \cref{main:minimal-surfaces} is necessary as shown in \cref{rem:dim-restriction-is-necessary}.
    \item We present a slightly more general result involving the language of tangential structures, see \cref{thm:criterion-for-minimal-surface}. 
    It is possible to derive an analogous result for almost spin manifolds.
  \end{enumerate}
\end{remark}

\noindent\textbf{Outline of the argument.}
To illustrate the argument for \cref{main:doubling-conjecture-spin-and-nonspin} and \cref{main:doubling-conjecture-almost-spin}, we give the proof in the case that $M$ is spin and $\partial M$ is simply connected.
Without loss of generality, we may assume that $M$ is connected.
We embed disjoint generators of $\pi_1(M)$ which all have trivial normal bundle since $M$ is oriented.
Then we perform $1$-surgeries along tubular neighborhoods of these generators to obtain a manifold with trivial fundamental group.
We thus obtain a manifold $M'$ which is simply connected.

\medskip

By the surgery theorem for positive mean curvature of Lawson--Michelsohn \cite[Theorem 3.1]{lawsonmichelsohn_embedding-and-surrounding-with-positive-mean-curvature}, $M'$ admits a psc-metric $g'$ with strictly mean convex boundary.
Furthermore, the disjoint union $M^\op\amalg M'$ is spin-cobordant to $\partial M\times [0,1]$ relative to the boundary by \cite[Proposition 3.25]{actionofmcg}.
Since $\partial M$ is simply connected, for any psc-metric $g$ on $M^\op\amalg M'$, there exists a psc-metric on $\partial M\times[0,1]$ which agrees with $g$ in a neighborhood of the boundary, see \cref{prop:psc-from-doubles-to-cylinders}.
Now, given a psc-metric $g_{\double}$ on the double $\double M$, we can consider $g_{\double}\amalg g'$ on $\double M\amalg M'=M\cup M^\op\amalg M'$.
By the above observation, we obtain a psc-metric on $M\cup (\partial M\times[0,1])\cong M$ which agrees with $g'$ near the boundary, hence is strictly mean convex, see \cref{fig:finish_m1-intro}

  \begin{figure}[ht]
    \scalebox{.8}{
    \begin{tikzpicture}
      \node(0) at (0,0) {\includegraphics[width=.9\textwidth]{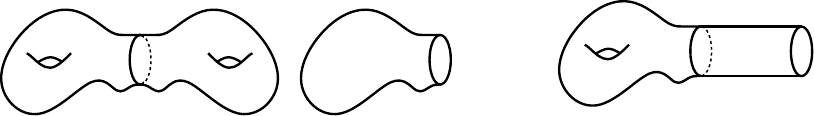}};
      \node(1) at (-5,1.2) {$M$};
      \node(1) at (4,1.2) {$M$};
      \node(2) at (-2.5,1.2) {$M^\op$};
      \node(3) at (-.5,1.1) {$M'$};
      \node(4) at (-1.6,-1.6) {$\underbrace{\quad\qquad\quad\qquad\qquad\qquad\qquad}_{\sim\  M^\op\,\amalg\, M\ \sim\  \partial M\times[0,1]}$};
      \node(6) at (1.4,0) {$\sim$};
    \end{tikzpicture}}
    \caption{}\label{fig:finish_m1-intro}
  \end{figure}

  \clearpage

The idea of this proof can be generalized to cover a larger class of manifolds, see \cref{lem:doubling-conjecture-criterion-one-component}.
However, the underlying principle can only work if the manifold $M$ has connected boundary.
In order to treat the case of several boundary components in \cref{lem:criterion-for-non-connected-boundaries}, we proceed as follows:
Given one component $\Sigma$ of the boundary, we find an area-minimal hypersurface $S$ inside the double $\double M$ of $M$, which is homologous to $\Sigma$.
Such an area-minimizer exists and is smooth in ambient dimension at most $11$ by \cite[Theorem 1.2]{Chodosh-Mantoulidis-Schulze-Wang_Generic-regularity-for-minimizing-hypersurfaces-in-dimension-11}.
A priori, submanifolds in the same homology class might still intersect, and we thus cannot assume that $\Sigma$ and $S$ bound a codimension 0 submanifold of $\double M$.
However, area-minimality of $S$ together with the monotonicity formula for arbitrary Riemannian metrics on the disk implies that a finite cover of $M$ contains disjoint copies of $\Sigma$ and $S$.
In contrast to the classical dimension-descent argument for scalar curvature using stable minimal hypersurfaces, we instead use the hypersurface $S$ as geometric interfaces along which we cut open this finite cover to obtain a manifold with a psc-metric and (stable) minimal boundary.
By \cite[Corollary 4.3]{Baer-Hanke_boundary-conditions-for-scalar-curvature} we can perturb the obtained metric so that $S$ becomes strictly mean convex and by a surgery argument we can make $\Sigma$ mean convex. 
Since the entire construction is local, we can treat one boundary component at a time to achieve a metric on $M$ with mean convex boundary.

\medskip

\noindent\textbf{Outline of the paper.}
For the proofs of our main results we use the language of tangential structures.
In \cref{sec:preliminaries} we gather the relevant notions, and we state the surgery and bordism principles for positive scalar and mean curvature using this language.
In the succeeding \cref{sec:criterion}, we give two criteria (\cref{lem:doubling-conjecture-criterion-one-component} and \cref{lem:criterion-for-non-connected-boundaries}) and one inheritance result (\cref{lem:inheritance-for-doubling-conjecture}) for the \cref{doubling-conjecture} which together form the basis for the proof of \cref{main:doubling-conjecture-spin-and-nonspin} and \cref{main:doubling-conjecture-almost-spin}, which is presented in \cref{sec:proof-of-main}.
Here, we also give the proof of the $4$-dimensional result from \cref{main:dimension-4}.
In \cref{sec:separating-hypersurfaces} we discuss \cref{que:minimal-hypersurface}, presenting examples and the proof of \cref{main:minimal-surfaces}.

\medskip

\noindent\textbf{Acknowledgements.} I thank Bernhard Hanke for numerous discussions about the doubling conjecture, minimal surfaces and scalar curvature cobordism. I would also like to thank Dieter Kotschick and Thomas Schick for their interest in this project. This work was partially supported by the DFG-SPP 2026 \enquote{Geometry at infinity} and by the University of Göttingen.

\section{Preliminaries}\label{sec:preliminaries}
\subsection{Tangential structures}

Our proofs are based on surgery principles for positive scalar and mean curvature.
The language of \emph{tangential structures} yields a convenient setting for working with these.
In order to define tangential structures, let $\bo(n)$ denote the classifying space for vector bundles of rank $n$, that is we have a bijection
\[\frac{\{V\to X \text{ vector bundle of rank } n\}}{\text{isomorphism}}\overset{{1:1}}\longleftrightarrow\frac{\{X\to \bo(n)\text{ continuous map}\}}{\text{homotopy}}\]
for every paracompact space $X$ (\cite[Theorem 1.16]{Hatcher_vector-bundles}).
We call the map $f\colon X\to\bo(n)$ corresponding to a vector bundle $V\to X$ the \emph{classifying map for $V$}.

\medskip

Note, that the classifying map for $V\oplus\underline\bbR\to X$ is given by the composition of the classifying map $f\colon X\to\bo(n)$ of $V\to X$ composed with the natural map $\bo(n)\to\bo(n+1)$.
The map $X\to \bo(n)\to\bo \coloneqq\colim_{n\to\infty}\bo(n)$ into the colimit of these natural maps is called the \emph{stable classifying map for $V$}.
The space $\bo$ classifies \emph{stable vector bundles}, see \cite[Section 1.2.1]{gollinger}.
Note, that a map $f\colon X\to \bo$ from a compact space lands in a finite stage $\bo(n)$ of this colimit and hence corresponds to (the stable isomorphism class of) a vector bundle.

\begin{defn}[Tangential structures]
  A \emph{(stable) tangential structure} is a map $\theta\colon B\to\bo$.
  Given a tangential structure $\theta$ and a vector bundle $V\to X$ with (stable) classifying map $f\colon X\to\bo$, a \emph{$\theta$-structure on $V$} is given by a lift $\ell$ of $f$ along $\theta$, that is $f=\theta\circ\ell$.
\end{defn}

\begin{example}
  Standard examples for tangential structures are $B=\bso(n)$ (orientations), $B=\bspin(n)$ (spin-structures) and $B=\ast$ (framings).
\end{example}

For studying positive scalar curvature, the most relevant tangential structures are \emph{tangential $2$-types} which we define next.

\begin{defn}[Tangential $2$-types]\label{def:tangential-2-types}
  Let $M$ be a manifold.
  A tangential structure $\theta\colon B\to \bo$ is called the \emph{(stable) tangential $2$-type of $M$} if the tangent bundle $TM\to M$ admits a $2$-connected $\theta$-structure $\ell$ and if $\theta$ is $2$-coconnected.%
  \footnote{That is $\ell$ induces isomorphisms on $\pi_0$ and $\pi_1$ and an epimorphism on $\pi_2$, whereas $\theta$ induces a monomorphism on $\pi_2$ and isomorphisms on $\pi_m$ for $m\ge 3$.}
\end{defn}

Tangential $2$-types always exist, as they are given by the second stage of the Moore-Postnikov tower for the classifying map $\tau\colon M\to\bo$ of the stable tangent bundle $TM$ of $M$.
Before giving the relevant examples of tangential $2$-types, we recall the following definition.
         
\begin{defn}\label{def:almost-spin-totally-nonspin}
  An orientable manifold $M$ is called
  \begin{enumerate}
    \item \emph{almost spin}, if $M$ does not admit a spin structure but its universal cover does.
    \item \emph{totally nonspin}, if its universal cover does not admit a spin structure.
  \end{enumerate}
\end{defn}

\noindent
Let us now give explicit descriptions of the relevant tangential $2$-types.

\begin{example}[Tangential $2$-types of orientable manifolds]\label{ex:tangential-two-types-of-manifolds}
  Let $M$ be a connected, orientable manifold of dimension $d\ge2$.
  \begin{enumerate}
    \item If $M$ is spin, then its tangential $2$-type is given by the composition of the projection and the natural map $\bspin\times B\pi_1(M)\to\bspin\to\bo$.
    \item If  $M$ is totally nonspin, then its tangential $2$-type is given by the composition of the projection and the natural map $\bso\times B\pi_1(M)\to\bo$.
    \item If $M$ is almost spin, then let $B$ be defined as the (homotopy) pullback
    \[\begin{tikzcd}
      B \arrow[r] \arrow[d]
          \arrow[dr, phantom, very near start, "{ \lrcorner }"]
        & B\pi_1(M) \arrow[d, "m"] \\
      \bso \arrow[r, "w_2"]
        & K(\bbZ/2,2)
    \end{tikzcd}\]
    Here, $w_2\in H^2(\bso;\bbZ/2)\cong[\bso,K(\bbZ/2,2)]$ is the universal second Stiefel--Whitney class and $m$ is the unique cohomology class such that $u^*m = w_2(TM)$ for $u\colon M\to B\pi_1(M)$ the map inducing the identity on fundamental groups \cite[p. 713]{kreck_surgery-and-duality}, which is unique up to homotopy.

    \medskip

    The tangential $2$-type of $M$ is given as the composition $\theta\colon B\to\bso\to\bo$ and a $\theta$-structure is given by an orientation and a map to $B\pi_1(M)$ such that the respective compositions with $m$ and $w_2$ are homotopic.
  \end{enumerate}
\end{example}

\noindent
We have the following characterization.

\begin{prop}\label{prop:criterion-for-spin-and-nonspin}
  Let $M$ be a connected, oriented manifold of dimension at least $5$.
  \begin{enumerate}
    \item $M$ is spin if and only if for every embedded surface $S\subset M$ the restriction $TM|_{S}$ of the tangent bundle to $S$ is trivial.
    \item $M$ is totally nonspin if and only if there exists an embedded sphere $S^2\subset M$ such that the restriction $TM|_{S^2}$ of the tangent bundle $TM$ to $S^2$ is nontrivial.
  \end{enumerate}
\end{prop}

\begin{remark}\label{rem:criterion-for-spin-and-nonspin-dimension-4}
  If $\dim(M)=4$, the statement is still correct if \emph{embedded} surfaces and spheres are replaced by \emph{immersed} surfaces and spheres.
\end{remark}

\begin{proof}[Proof of \cref{prop:criterion-for-spin-and-nonspin}]
  Let $w_2(TM)\in H^2(M;\bbZ/2)$ denote the second Stiefel--Whitney class of $TM$.
  Since the pairing of cohomology and homology with $\bbZ/2$-coefficients is perfect, we have that $w_2(TM)=0$ if and only if $\scpr{w_2(TM);x}=0$ for every class $x\in H_2(M;\bbZ/2)$.
  \begin{enumerate}
    \item If $M$ is spin, and $\iota\colon S\embeds M$ is an embedding of a surface, then $w_i(TM|_S)=\iota^*w_i(TM)=0$ for $i=1,2$.
    Therefore, the classifying map $S\to\bo$ of $TM|_{S}$ lifts along $\bspin\to\bo$, which is $2$-connected if $d\ge2$.
    Since $S$ is 2-dimensional, the map $S\to\bspin$ factors through a point up to homotopy, for example by \cite[Lemma 4.6]{hatcher_at}, and thus $TM|_S$ is trivial.

    \medskip

    On the other hand, every class $x\in H_2(M;\bbZ/2)$ is represented by $\iota_*[S]$ for a smooth map $\iota\colon S\embeds M$ from a surface $S$ (see \cite[Théorème III.2]{thom_quelques-proprietes-globales-des-varietes-differentiable}).
    This map can be chosen to be an embedding by the Whitney embedding theorem, because $\dim(M)\ge5$ (resp. an immersion if $\dim(M)=4$).
    Since
    \[\scpr{w_2(TM),\iota_*[S]} = \scpr{\iota^*w_2(TM),[S]} = \scpr{w_2(TM|_S),[S]} = 0.\]
    By the above argument, we deduce that $w_2(TM)=0$.

    \medskip

    \item Let $\pi\colon \widetilde{M}\to M$ be the universal cover.
    If $M$ is totally nonspin, there is an embedding of a surface $\tilde \iota\colon S\embeds \widetilde M$ such that 
    \[\scpr{w_2(T\widetilde{M}|_{\tilde\iota(S)}),[S]} = \scpr{w_2(T\widetilde{M}),\tilde\iota_*[S]}\not=0\]
    by part (i).
    Since $\widetilde M$ is simply connected, the Hurewicz-homomorphism $\pi_2(\widetilde M)\to H_2(\widetilde M;\bbZ)$ is an isomorphism and $H_2(\widetilde{M};\bbZ/2)\cong H_2(\widetilde{M};\bbZ)\otimes\bbZ/2$ by the universal coefficient theorem.
    Therefore, we may assume that $S$ is a $2$-sphere.
    Since $T\widetilde{M} \cong \pi^*TM$ and covering maps are local diffeomorphisms, $\iota\coloneqq\pi\circ \tilde\iota$ is an immersion and by the Whitney embedding theorem it can be changed by a homotopy to an embedding $j\colon S^2\embeds M$.
    Since 
    \[w_2(TM|_{j(S^2)}) = w_2(j^*TM) = w_2(\iota^*TM) = w_2(\tilde \iota^*T\widetilde M) = w_2(T\widetilde M|_{\tilde\iota(S^2)})\not=0,\]
    we find that $j^*TM = TM|_{j(S^2)}$ is nontrivial.

    \medskip
    
    On the other hand, if there exists some embedding $\iota\colon S^2\embeds M$ such that $\iota^*TM$ is nontrivial, we can lift this to an embedding $\tilde\iota\colon S^2\embeds \widetilde M$ by the lifting property for coverings.
    Since
    \[T\widetilde M|_{\tilde\iota(S^2)}\cong\tilde\iota^*T\widetilde M = \tilde\iota^*\pi^*TM = \iota^*TM = TM|_{\iota(S^2)},\]
    we observe that $T\widetilde M|_{\tilde\iota(S^2)}$ is nontrivial and by (i), $\widetilde M$ is not spin.\qedhere
  \end{enumerate}
\end{proof}

Finally, we introduce the notion of \emph{extendable tangential $2$-type}, which will play a central role in the succeeding section.

\begin{defn}\label{def:extendable-tangential-2-type}
  Let $M$ be a manifold and let $\Sigma\subset M$ be a hypersurface.
  Let $\theta\colon B\to\bo$ be the tangential $2$-type of $\Sigma$.
  We say, that the \emph{tangential $2$-type of $\Sigma$ extends to $M$}, if there exist $\theta$-structures $\ell_\Sigma\colon \Sigma\to B$ and $\ell_M\colon M\to B$ such that
  \begin{enumerate}
    \item $\ell_\Sigma$ is $2$-connected
    \item $\ell_\Sigma$ is homotopic to $\ell_M\circ\iota$, where $\iota\colon \Sigma\to M$ denotes the inclusion.
  \end{enumerate}
\end{defn}

\begin{remark}\label{rem:on-extendable-tangential-2-type}
  \begin{enumerate}
    \item If $\Sigma\subset\partial M$ is a component such that $\Sigma\embeds M$ is $2$-connected, then the tangential $2$-type of $\Sigma$ extends to $M$.
    This follows from \cite[Lemma 4.6]{hatcher_at}, because $(M,\Sigma)$ is homotopy equivalent to a (relative) $CW$-complex with no cells of dimension $0$, $1$ and $2$ and $B\to\bo$ is $2$-coconnected.
    \item 
    If $\Sigma=\partial M$ and the tangential $2$-type of $\Sigma$ extends to $M$, then $M$ is a $\theta$-nullbordism of $\Sigma$ and $M^\op\amalg M$ is $\theta$-bordant to $\Sigma\times[0,1]$, see \cite[Proposition 3.25]{actionofmcg}.
    \item 
    If the tangential $2$-type of $\Sigma$ extends to $M$, then the map $\pi_0(\Sigma)\to\pi_0(M)$ induced by the inclusion is injective.
    Assuming that $M$ is connected, this forces $\Sigma$ to be connected as well.
  \end{enumerate}
\end{remark}

The following proposition provides example cases, where the tangential $2$-type of a hypersurface $\Sigma$ extends to $M$. 

\begin{prop}\label{prop:tangential-2-type-extends}
  Let $M$ be an oriented manifold and let $\iota\colon\Sigma\embeds M$ be an embedding of a connected hypersurface with trivial normal bundle such that the induced map $i\coloneqq \iota_\ast\colon \pi_1(\Sigma)\to\pi_1(M)$ is split-injective.
  Assume one of the following holds:
  \begin{enumerate}
    \item $M$ is spin
    \item $\Sigma$ is totally nonspin
    \item $M$ is almost spin and the map $Bi^*\colon \cohomology^2\bigl(B\pi_1(M);\bbZ/2\bigr)\to \cohomology^2\bigl(B\pi_1(\Sigma);\bbZ/2\bigr)$ is an isomorphism.
  \end{enumerate}
  Then the tangential $2$-type of $S$ extends to $M$.
\end{prop}

\begin{proof}\leavevmode
  \begin{enumerate}
    \item Let $f\colon M\to\bspin$ be a lift of the classifying map for the stable tangent bundle of $M$.
    Since the normal bundle of $\Sigma$ is trivial, the restriction of $f$ to $\Sigma$ is a lift of the classifying map for the stable tangent bundle of $\Sigma$.
    Therefore, $\Sigma$ is spin and the tangential $2$-type of $\Sigma$ is given by 
    \[\theta\colon \bspin\times B\pi_1(\Sigma)\to\bo.\]
    Let $u_\Sigma\colon \Sigma\to B\pi_1(\Sigma)$ be the map inducing the identity on fundamental groups (which is unique up to homotopy) and define $\ell_\Sigma = (f|_\Sigma, u_\Sigma)\colon\Sigma\to\bspin\times B\pi_1(\Sigma)$, which is an isomorphism on $\pi_0$ and $\pi_1$ and surjective on $\pi_2$, hence $2$-connected.

    \medskip

    Let $s\colon \pi_1(M)\to\pi_1(\Sigma)$ be a split of the map $i$ induced by the inclusion and let $u\colon M\to B\pi_1(M)$ be the map which induces the identity on fundamental groups.
    We define $\ell_M\coloneqq (f,Bs\circ u)$.
    Since maps from a connected $CW$-complex $X$ into $B\pi_1(M)$ are determined uniquely up to homotopy by the homomorphism $\pi_1(X)\to \pi_1(M)$ (\cite[Proposition 1B.9]{hatcher_at}) and $u \circ \iota$ and $Bi\circ u_\Sigma$ both induce the map $i=\iota_\ast$ on fundamental groups, they are homotopic.
    We thus obtain
    \[Bs\circ u \circ \iota \sim Bs\circ Bi\circ u_\Sigma \sim u_\Sigma,\]
    because $Bs$ is a split of $B\iota$ as well.
    Therefore, $\ell_\Sigma$ and $\ell_M\circ\iota$ are homotopic.

    \medskip

    \item If $\Sigma$ is totally nonspin, then so is $M$ by \cref{prop:criterion-for-spin-and-nonspin}.
    Choosing an orientation on $M$, we get a map $f\colon M\to\bso$, and as the normal bundle of $\Sigma$ is trivial, the restriction $f|_\Sigma\colon \Sigma\to\bso$ yields an orientation on $\Sigma$.
    We define $u_\Sigma$, $u$, $\ell_\Sigma$ and $\ell_M$ as in the spin-case with $\bspin$ replaced by $\bso$ and the proof is finished verbatim.

    \medskip

    \item First, we show that $\Sigma$ is almost spin, too.
    Note, that $\Sigma$ cannot be totally nonspin, for if it was, there would be an immersed $2$-sphere in $N_i$ with non-trivial normal bundle by \cref{prop:criterion-for-spin-and-nonspin} and \cref{rem:criterion-for-spin-and-nonspin-dimension-4}.
    This immersed $2$-sphere would be homotopic to an embedded $2$-sphere inside the collar by the Whitney embedding theorem, which still has non-trivial normal bundle.
    But $M$ was assumed to be almost spin and hence contains no embedded $2$-sphere with non-trivial normal bundle.

    \medskip

    In order to show that $\Sigma$ is nonspin, we choose a split $s\colon \pi_1(M)\to \pi_1(\Sigma)$ of $i$.
    Since $M$ is almost spin we have $w_2(TM)=u^*m\not=0$ for $u\colon M\to B\pi_1(M)$ and $m \in \cohomology^2(B\pi_1(M);\bbZ/2)$ as in \cref{ex:tangential-two-types-of-manifolds}.
    Let $u_\Sigma\colon \Sigma\to B\pi_1(\Sigma)$ be the map inducing the identity on fundamental groups. 
    As before, we have $u\circ\iota\sim Bi\circ u_\Sigma$.
    Hence,
    \begin{equation}\label{eq:w2}
      w_2(T\Sigma) = \iota^*w_2(TM) = \iota^*u^*m = u_\Sigma^* Bi^* m.
    \end{equation}
    The induced map $u_\Sigma^*$ is injective on second cohomology by \cite[p. 713]{kreck_surgery-and-duality}.
    Since $m\not=0$ and $Bi^*$ is an isomorphism by assumption, $w_2(T\Sigma)$ is nonzero and hence $\Sigma$ is almost spin.
    Note, that by the defining property $u_\Sigma^*m_\Sigma=w_2(T\Sigma)$, we have $Bi^*m = m_\Sigma\in \cohomology^2(B\pi_1(\Sigma);\bbZ/2)$.

    \medskip

    Now, let $\theta$ be the tangential $2$-type of $\Sigma$.
    A $\theta$-structure on $M$ is given by an orientation $o\colon M\to\bso$ and a map $v\colon M\to B\pi_1(\Sigma)$ such that $m_\Sigma\circ v$ is homotopic to $w_2\circ o$ for $m_\Sigma\in \cohomology^2(B\pi_1(\Sigma);\bbZ/2)$ as above and $w_2\in \cohomology^2(\bso;\bbZ/2)$ the universal second Stiefel--Whitney class.
    Let $o$ be the orientation on $M$ coming from its almost spin structure and let $o|_\Sigma$ be the induced orientation on $\Sigma$.
    We define $\ell_\Sigma\coloneqq (o|_{\Sigma}, u_\Sigma)$.
    By \eqref{eq:w2} this defines a $\theta$-structure on $\Sigma$.

    \medskip

    In order to define an extension of $\ell_\Sigma$ to $M$, let  $u\colon M\to B\pi_1(M)$ be the map inducing the identity on fundamental groups and let $s\colon\pi_1(M)\to \pi_1(\Sigma)$ be a split of $i = \iota_\ast$.
    Since $Bi^*$ is an isomorphism on $\cohomology^2(-;\bbZ/2)$, the map $Bs^*$ is a two-sided inverse for $Bi^*$ on $\cohomology^2(-;\bbZ/2)$.\footnote{Since $i$ is assumed to be split-injective we would only get $Bs\circ Bi=\id$ and hence $Bi^*Bs^*=\id$ without the additional assumption on the induced map in second cohomology.}
    Therefore, we have
    \[w_2(TM) = u^*m = u^*Bs^*Bi^*m = (Bs\circ u)^*m_\Sigma = v^*m_\Sigma,\]
    where we used $Bi^*m = m_\Sigma$.
    We obtain a $\theta$-structure $\ell_M\coloneqq (o,Bs\circ u)$ on $M$ which satisfies $\ell_M\circ\iota \sim \ell_\Sigma$, because $Bs\circ u\circ\iota$ and $u_\Sigma$ both induce the identity on fundamental groups.\qedhere
  \end{enumerate}
\end{proof}

\subsection{Surgery principles for scalar and mean curvature}

With the language of tangential structures at hand, we can translate the Gromov--Lawson--Schoen--Yau surgery theorem for positive scalar curvature into the cobordism setting by adhering to the handle cancellation technique from the proof of the $h$- and $s$-cobordism theorem (see for example \cite{smale_on-the-structure-of-manifolds}, \cite{milnor_lectures_on_the_h_cobordism_theorem} or \cite{wall_geometric-connectivity}). 
This gives the following result:
\begin{theorem}[{\cite[Theorem 1.5]{ebertfrenck},\cite{gl80a},\cite{schoenyau_classical}}]\label{thm:cobordism-theorem-for-psc}
  Let $M_0,M_1$ be closed manifolds of dimension at least $5$.
  Let furthermore $W\colon M_0\leadsto M_1$ be a cobordism from $M_0$ to $M_1$ such that the tangential $2$-type of $M_1$ extends to $W$.
  Then, if $M_0$ admits a psc-metric $g_0$, there is a psc-metric $g_1$ on $M_1$ as well.
\end{theorem}

\begin{remark}\label{rem:extension-of-surgery-principle}
  This theorem has a few improvements:
  \begin{enumerate}
    \item It also holds for manifolds $M_0$, $M_1$ with boundary if the cobordism is trivial on the boundary: In this case one can again perform handle cancellation (see \cite{wall_geometric-connectivity}) and all surgeries take place away from the boundary.
    Since the construction in \cite{gl80a} only changes the metric locally near the surgery embeddings (see also \cite[Theorem 6.1]{FrenckHankeHirsch2026} for a quantitative statement), the metric remains unchanged in a neighborhood of the boundary.
    \medskip
    \item If the cobordism $W$ contains a trivial cobordism $A\times[0,1]$ for a closed set $A\subset M_0$, we can assume that the psc-metrics $g_0$ and $g_1$ agree on $A$.
    Again, this follows because the construction in \cite{gl80a,FrenckHankeHirsch2026} is local.
    \medskip
    \item If the inclusion $M_1\embeds W$ is already $2$-connected, then $W$ admits a psc-metric which extends $g_0$ on $M_0$ and is of product type near both boundaries.
    This is a result of Gajer \cite{gajer}.
    It can also be derived from the main result of \cite{ebertfrenck} or \cite{chernysh_on_the_homotopy_type_of_the_space}, see \cref{prop:extension-to-cobordism}.
    
    \medskip

    It is always possible to make the embedding $M_1\embeds W$ $2$-connected by performing surgeries on the interior of $W$, see \cite[Appendix B]{hebestreitjoachim_twisted-spin-cobordism-and-positive-scalar-curvature} or \cite[Proposition 6.3]{ebertfrenck}.
    \end{enumerate} 
\end{remark}

Let us consider the special case of doubles of cobordisms.
Let $M\colon S_0\leadsto S_1$ be a $\theta$-cobordism for $\theta\colon B\to \bo$ the tangential $2$-type of $S_0$.
Then there is a $\theta$-structure on $M^\op\colon S_1\leadsto S_0$ such that the double $M\cup_{S_1} M^\op$ is $\theta$-cobordant to $S_0\times[0,1]$ relative to the boundary (see \cite[Proposition 3.25]{actionofmcg}).
Hence, we obtain the following:

\begin{prop}\label{prop:psc-from-doubles-to-cylinders}
  Let $M\colon S_0\leadsto S_1$ be a cobordism of dimension at least $5$.
  Assume that the tangential $2$-type $\theta$ of $S_0$ extends to $M$ and let $W\colon S_0\leadsto S_0$ be $\theta$-cobordant to the double $M\cup_{S_1} M^\op$ relative to the boundary.
  If $W$ admits a psc-metric $g$, then $S_0\times[0,1]$ admits a psc-metric $g'$ such that $g$ and $g'$ agree in a neighborhood of the boundary.
\end{prop}
\noindent Now, we turn to the corresponding surgery principle for positive mean curvature.

\begin{theorem}[{\cite[Theorem 3.1]{lawsonmichelsohn_embedding-and-surrounding-with-positive-mean-curvature}}]\label{thm:lawson-michelsohn}
  Let $\Sigma$ be a (normally oriented) strictly mean convex hypersurface in a Riemannian manifold $M$.
  If $\Sigma'$ is obtained from $\Sigma$ by attaching a $p$-handle inside $M$ to the positive side of $\Sigma$ for $p\le \dim(M)-2$, then $\Sigma'$ can be chosen to be strictly mean convex.
\end{theorem}

The following corollary translates this into an external statement. Its proof is very similar to the proof of \cite[Theorem 5.1]{lawsonmichelsohn_embedding-and-surrounding-with-positive-mean-curvature}. Note that there are no assumptions on the dimension of $W$ here.

\begin{cor}\label{lem:cobordism-principle-for-hpsc-bounding}
  Let $M\colon S\leadsto \Sigma$ be a cobordism that admits a metric of positive scalar curvature and strictly mean convex boundary $\Sigma$.
  If $M'\colon\Sigma\leadsto\Sigma'$ is another cobordism that only consists of handles of codimension at least $2$, then $M\cup_\Sigma M'$ also admits a psc-metric with strictly mean convex boundary $\Sigma'$.
\end{cor}

\begin{proof}
  By Gromov's (relative) $h$-principle, there exists a psc-metric on $M\cup_\Sigma M'$ which extends the given psc-metric on $M$.
  Note, that not this extended metric may not have mean convex boundary.
  Since $M'$ only consists of handles of index at most $2$, \cref{thm:lawson-michelsohn} implies that there is an embedding $M'\embeds M\cup_\Sigma M'$ whose incoming boundary agrees with $\Sigma$ and whose outgoing boundary is strictly mean convex.
  Furthermore, the union of $M$ and this embedded copy of $M'$ has positive scalar curvature (as a codimension $0$ submanifold of $M\cup M'$) and strictly mean convex boundary.
\end{proof}

The condition on the codimension of handles can again be translated using handle cancellation.
Since we only want to cancel $0$- and $1$-handles, this also works for $5$-dimensional manifolds by the following result of Wall.

\begin{theorem}[{\cite[Theorem 3]{wall_geometric-connectivity}}]\label{thm:wall-geometric-connectivity}
  Let $M\colon S_0\leadsto S_1$ be a cobordism such that the inclusion $S_1\embeds M$ is $r$-connected for $r\le \dim(M)-4$.
  Then $M$ contains no handles of index $\ge \dim(M)-r$.\footnote{\cite[Thorem 3]{wall_geometric-connectivity} states that there are no handles of index $\le r$ if $S_0\embeds M$ is $r$-connected.
  We use this statement for $S_1$ instead and turning $M$ upside down turns an $s$-handle into a $(\dim(M)-s)$-handle which yields this statement.}
\end{theorem}

Combining \cref{lem:cobordism-principle-for-hpsc-bounding} and \cref{thm:wall-geometric-connectivity} we obtain the following consequence, see {\cite[Theorem 5.1]{lawsonmichelsohn_embedding-and-surrounding-with-positive-mean-curvature}}.

\begin{cor}\label{cor:boundaries-with-hpsc-metrics}
  Let $M$ be a manifold with boundary of dimension $d\ge5$ such that $\partial M\embeds M$ is $1$-connected.
  Then $M$ admits a metric of positive scalar curvature with strictly mean convex boundary.
\end{cor}

We can use \cref{lem:cobordism-principle-for-hpsc-bounding} and \cref{thm:wall-geometric-connectivity} to show that every (nullbordant) oriented manifold is the mean convex boundary of some Riemannian manifold of positive scalar curvature, see \cref{thm:every-manifold-bounds-hpsc-metrics}.
This can be rephrased to the statement that every (nullbordant) oriented manifold carries a Riemannian metric which admits a fill-in with positive scalar and mean curvature.

\section{Criteria for the doubling conjecture}\label{sec:criterion}
In this section we present two criteria (\cref{lem:doubling-conjecture-criterion-one-component} and \cref{lem:criterion-for-non-connected-boundaries}) and an inheritance result (\cref{lem:inheritance-for-doubling-conjecture}) for the \cref{doubling-conjecture}.
These form the basis for the proof of \cref{main:doubling-conjecture-spin-and-nonspin} and \cref{main:doubling-conjecture-almost-spin}.

\begin{lemma}\label{lem:doubling-conjecture-criterion-one-component}
  Let $M$ be a manifold of dimension at least $5$ such that the tangential $2$-type of $\partial M$ extends to $M$.
  Then the \cref{doubling-conjecture} holds for $M$.
\end{lemma}

\begin{remark}
  Considering one component of $M$ at a time, it suffices to study the \cref{doubling-conjecture} for connected manifolds.
  As mentioned in \cref{rem:on-extendable-tangential-2-type}, the assumption from \cref{lem:doubling-conjecture-criterion-one-component} then dictates that $\partial M$ is connected, as well.
  In conjunction with \cref{lem:inheritance-for-doubling-conjecture} or by employing \cref{lem:criterion-for-non-connected-boundaries} we obtain an approach to the \cref{doubling-conjecture} for manifolds with non-connected boundary.
\end{remark}

\begin{proof}[Proof of \cref{lem:doubling-conjecture-criterion-one-component}]
  Let $g_\double$ be a psc-metric on the double $\double M$ of $M$.
  Let $\theta\colon B\to\bo$ be the tangential $2$-type of $\Sigma$ and let $\ell_\Sigma$, $\ell_M$ be as in \cref{def:extendable-tangential-2-type}.
  The manifold $M$ is then $\theta$-cobordant to a manifold $M'$ whose $\theta$-structure $\ell'\colon M'\to B$ is $2$-connected.
  Since $\ell_\Sigma\colon\Sigma\to B$ is $2$-connected, it follows that the inclusion $\Sigma\embeds M'$ induces an isomorphism on $\pi_0$ and $\pi_1$ and is hence $1$-connected.
  Therefore, $M'$ admits a metric $g'$ of positive scalar curvature with strictly mean convex boundary by \cref{cor:boundaries-with-hpsc-metrics}.
  The disjoint union of metrics $g_\double\amalg g'$ on $\double M\amalg M'$ has positive scalar curvature and mean convex boundary.
  \begin{figure}[ht]
    \begin{tikzpicture}
      \node(0) at (0,0) {\includegraphics[width=.9\textwidth]{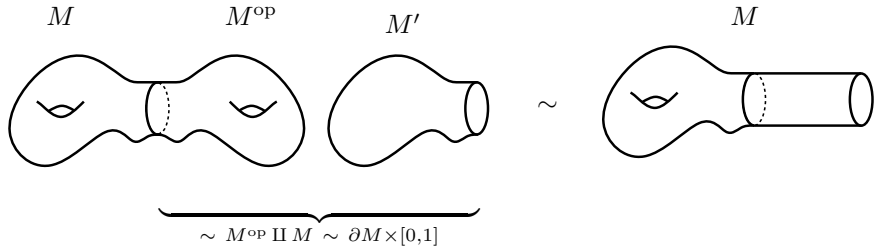}};
      \node(1) at (-5,1.2) {$M$};
      \node(1) at (4,1.2) {$M$};
      \node(2) at (-2.5,1.2) {$M^\op$};
      \node(3) at (-.5,1.1) {$M'$};
      \node(4) at (-1.6,-1.6) {$\underbrace{\quad\qquad\quad\qquad\qquad\qquad\qquad}_{\sim\  M^\op\,\amalg\, M\ \sim\  \partial M\times[0,1]}$};
      \node(6) at (1.4,0) {$\sim$};
    \end{tikzpicture}
    \caption{Constructing a psc-metric on $M$ with strictly mean convex boundary.}\label{fig:finish_m1}
  \end{figure}
  Since $M'$ is $\theta$-cobordant to $M$, $M^\op\amalg M'$ is $\theta$-cobordant to the double $M^\op\amalg M$ of $M$ relative to the boundary.
  By \cref{prop:psc-from-doubles-to-cylinders} there exists a psc-metric $g''$  on $\partial M\times[0,1]$ which agrees with $g_\double|_{M^\op}$ in a neighborhood of the incoming boundary and its outgoing boundary is mean convex.
  Thus, we can extend $g''$ by $g_\double|_M$ to a psc-metric on $M\cup \partial M\times[0,1]\cong M$ with mean convex boundary, see \cref{fig:finish_m1}.
\end{proof}

If the dimension of $M$ is at least $6$, we can adhere to \cref{thm:cobordism-theorem-for-psc} instead of \cref{cor:boundaries-with-hpsc-metrics} and derive the existence of a psc-metric on $M$ of product type.

\begin{cor}\label{cor:doubling-for-extendable-tangential-two-type}
  Let $M$ be a manifold of dimension at least $6$ with boundary $\partial M$ such that the tangential $2$-type of $\partial M$ extends to $M$.
  If $\double M$ admits a psc-metric, then $M$ admits a psc-metric that is of product type near the boundary.
\end{cor}

In contrast to \cref{lem:doubling-conjecture-criterion-one-component}, the following criterion is applicable to manifolds with non-connected boundary but has a dimensional constraint.

\begin{lemma}\label{lem:criterion-for-non-connected-boundaries}
  Let $M$ be an oriented manifold with boundary $\partial M$ satisfying $5\le \dim(M)\le 11$.
  If the tangential $2$-type of every component of $\partial M$ extends to $M$, then the \cref{doubling-conjecture} holds for $M$.
\end{lemma}

We need some preparation for the proof of \cref{lem:criterion-for-non-connected-boundaries}.
We start by giving an analogue of the monotonicity formula for area-minimizing hypersurfaces of Riemannian manifolds with bounded sectional curvature.

\begin{lemma}\label{lem:monotonicity-for-bounded-curvature}
  Let $g$ be a Riemannian metric on $D^{n+1}$ with $-\kappa\le\sec_g\le \kappa$ for a constant $\kappa>0$ and let $R\coloneqq\min(\injrad(g),\tfrac{1}{2\sqrt{\kappa}})$.
  Furthermore, Let $S\subset \bbR^{n+1}$ be an area-minimizing hypersurface with respect to this metric which passes through the origin and satisfies $\partial S\cap B(0,R)=\emptyset$.
  Then, for $f(r)\coloneqq r^n\cdot\exp(-\sqrt{\kappa}n^2r)$ the function
  \[\alpha\colon r\mapsto \frac{\vol\bigl(B(0,r)\cap S\bigr)}{f(r)}\]
  is monotonously increasing on $(0,R)$.
  For $r\to0$ we furthermore have
  \[\alpha(r)\to \vol_{\eucl}\bigl(B(0,1)\cap\bbR^n\times\{0\}\bigr)\eqqcolon w_n\]
  In particular, we have for $r<R$
  \[\vol\bigl(B(0,r)\cap S\bigr)\ge w_n\cdot f(r)\]
  and for $r \le \tfrac12\min(\injrad(g),\tfrac1{n^2\sqrt{\kappa}})$ we get
  \begin{equation}\label{eq:uniform-volume-estimate}
    \vol\bigl(B(0,r)\cap S\bigr)\ge w_n f(r) = w_n r^n \exp\left(-n^2\sqrt{\kappa}r\right)\ge \exp(-1)w_nr^n.
  \end{equation}
\end{lemma}

\begin{proof}
  Let $r\in(0,R)$ and let us introduce for $\rho\in(0,r]$ the following notation
  \begin{align*}
    S_\rho&\coloneqq S\cap B(0,\rho)  &\partial S_\rho&\coloneqq S\cap \partial B(0,\rho)\\
    C_\rho&\coloneqq \cone(0,\partial S_r)\cap B(0,\rho)  &\partial C_\rho&\coloneqq C_\rho\cap \partial B(0,\rho)
  \end{align*}
  \begin{figure}[ht]
    \scalebox{.6}{
    \begin{tikzpicture}
      \node at (0,0) {\includegraphics[width=.6\textwidth]{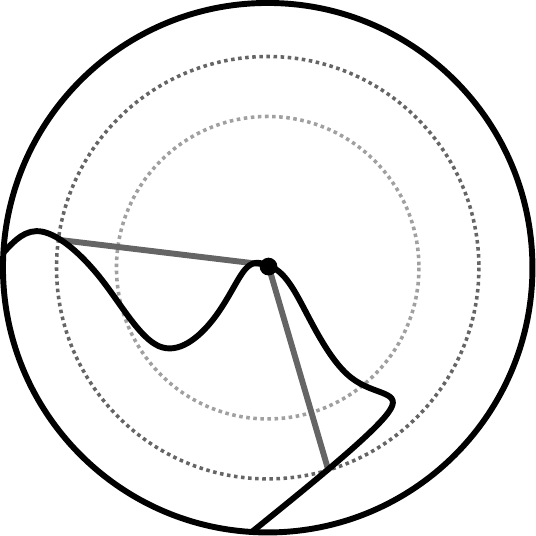}};
      \node at (1,-.6) {$S$};
      \node at (-1.6,-.1) {$C_r$};
      \node[rotate=-7] at (-1,0.6) {$\overbrace{\qquad\qquad\quad\ \ }^{C_\rho}$};
      \node at (0,3.3) {$\partial B(0,R)$};
      \node at (0,2.5) {$\partial B(0,r)$};
      \node at (0,1.7) {$\partial B(0,\rho)$};
    \end{tikzpicture}}
  \end{figure}

  \noindent Since the injectivity radius of $g$ is smaller than $R$, the metric $g$ is given by
  \[g=\drho^2 + g_\rho\]
  for a family of metrics $g_\rho$ on $S^{n-1}$.
  By \cite[Theorem 11.10]{Lee_Introduction-to-Riemannian-manifolds}, we have the following comparison to the metrics of constant curvature on $B(0,R)$
  \[\drho^2 + \frac1{{\kappa}}\sin(\sqrt{\kappa}\rho)^2g_\circ\eqqcolon g_{\kappa}\le g\le g_{-\kappa}\coloneqq\drho^2 + \frac1{{\kappa}}\sinh(\sqrt{\kappa}\rho)g_\circ,\]
  where $h\le h'$ means that $h(v,v)\le h'(v,v)$ for all $v\in TD^n$ and $g_\circ$ denotes the round metric.
  In particular, 
  \[\frac1{{\kappa}}\sin(\sqrt{\kappa}\rho)^2g_\circ\le g_\rho\le \frac1{{\kappa}}\sinh(\sqrt{\kappa}\rho)^2g_\circ\]
  Thus, for any $\rho\in(0,r]$ we have
  \[g_\rho\le \left(\frac{\sinh(\sqrt{\kappa}\rho)}{\sin(\sqrt{\kappa}r)}\right)^2g_r\]
  and in particular
  \begin{equation}\label{eq:area-scaling}
    \area(\partial C_\rho) = \left(\frac{\sinh(\sqrt{\kappa}\rho)}{\sin(\sqrt{\kappa}r)}\right)^{n-1}\area(\partial C_r)
  \end{equation}
  Since $S$ is area-minimizing, we have 
  \begin{align*}
    \vol(S_r)\le{}& \vol(C_r)\\
      ={}& \int_{0}^{r} \area(\partial C_\rho)\drho\\
      \overset{(\ref{eq:area-scaling})}{\le}{}& \int_{0}^{r} \left(\frac{\sinh(\sqrt{\kappa}\rho)}{\sin(\sqrt{\kappa}r)}\right)^{n-1}\area(\partial C_r)\drho\\
      ={}& \frac{\area(\partial C_r)}{\sin(\sqrt{\kappa}r)^{n-1}}\int_{0}^{r}\sinh(\sqrt{\kappa}\rho)^{n-1}\drho
  \end{align*}
  Since $\sinh(x) \le \exp(x)-1$ for every $x\ge0$ and 
  \[(\exp(\sqrt\kappa \rho)-1)^{n-1} = \tfrac1{n\sqrt\kappa}\dbyd{\rho} (\exp(\sqrt\kappa \rho)-1)^n\]
  for any $\rho\ge0$, we can estimate
  \begin{align*}
    \int_{0}^{r}\sinh(\sqrt{\kappa}\rho)^{n-1}\drho\le{}&\int_{0}^{r}\bigl(\exp(\sqrt{\kappa}\rho)-1\bigr)^{n-1}\drho\\
      ={}& \frac1{n\sqrt{\kappa}}\int_{0}^{r}\dbyd{\rho} \bigl(\exp(\sqrt{\kappa}\rho)-1\bigr)^n\drho\\
      ={}& \frac1{n\sqrt{\kappa}}\bigl(\exp(\sqrt{\kappa}r)-1\bigr)^n
  \end{align*}
  and thus
  \begin{align*}
    \vol(S_r)\le{}&\frac1{n\sqrt{\kappa}}\frac{(\exp(\sqrt{\kappa}r)-1)^n}{\sin(\sqrt{\kappa}r)^{n-1}}\area(\partial C_r).
  \end{align*}
  Furthermore, $\partial C_r = \partial S_r$ by definition and hence
  \begin{align*}
    \dbyd{r}\vol(S_r) ={}& \area(\partial S_r ) = \area(\partial C_r)\\
      \ge{}&n\sqrt{\kappa}\frac{\sin(\sqrt{\kappa}r)^{n-1}}{(\exp(\sqrt{\kappa}r)-1)^n}\vol(S_r).
  \end{align*}
  The functions $a(x) = \exp(x)-1-x-x^2$ and $b(x)=\sin(x)-x+x^3$ satisfy
  \begin{align*}
    a(0)&=0 &a'(0)&=0 & a''(x) &= \exp(x)-2<0\quad\text{for}\quad x\le\frac12\le \log(2)\\
    b(0)&=0 &b'(0)&=0 & b''(0) &= 0\qquad b'''(x)=-\cos(x)+6>0
  \end{align*}
  and thus $a(x)$ is non-positive and $b(x)$ is non-negative for $x\le 1/2$.
  Therefore, $\exp(x)-1\le x+x^2$ and $\sin(x)\ge x-x^3$.
  Furthermore, we have 
  \[\tfrac1{1+x}\ge 1-x \quad\text{and}\quad (1-x)^n\ge 1-nx \quad\text{for}\quad x\in(0,2),\]
  because $(1+x)(1-x) = 1-x^2\le 1$ and $c(x)=(1-x)^n-(1-nx)$ satisfies
  \begin{align*}
    c(0)&=0 &c'(x)&=n(1-x)^{n-1} + n = n(1+(1-x)^n)>0 \text{ for } x\in(0,2).
  \end{align*}
  Putting these together, we obtain 
  \begin{align*}
    \begin{split}
      \frac{\sin(x)^{n-1}}{(\exp(x)-1)^n}\ge{}& \frac{(x-x^3)^{n-1}}{(x+x^2)^n} = \frac{x^{n-1}(1-x)^{n-1}(1+x)^{n-1}}{x^n(1+x)^{n}}\\
        \ge{}& \frac{(1-x)^{n-1}}{x(1+x)} \ge \frac{(1-x)^n}{x} \ge \frac{1-nx}{x}\\
    \end{split}
  \end{align*}
  for $x\in(0,\tfrac{1}{2}]$.
  Thus, for $r\le1/(2\sqrt{\kappa})$, we get 
  \begin{align}\label{eq:fraction-sin-exponential}
    \begin{split}
      \dbyd{r}\vol(S_r) \ge{}&n\sqrt{\kappa}\frac{\sin(\sqrt{\kappa}r)^{n-1}}{(\exp(\sqrt{\kappa}r)-1)^n}\vol(S_r)\\
      \ge{}&n\frac{1-n\sqrt{\kappa}r}{r}\vol(S_r)
    \end{split}
  \end{align}
  We can now compute the $r$-derivative of $\alpha$:
  \begin{align*}
    \alpha'(r) ={}& \dbyd{r}\frac{\vol(S_r)}{f(r)}\\
      ={}& \frac1{f(r)}\left(\dbyd{r}\vol(S_r)\right) - \frac1{f(r)^2}f'(r)\vol(S_r)\\
      ={}& \frac1{f(r)^2}\left(f(r)\left(\dbyd{r}\vol(S_r)\right) - f'(r)\vol(S_r)\right)\\
      \ge{}& \frac1{f(r)^2}\left(nf(r)\frac{1-n\sqrt{\kappa}r}{r}\vol(S_r)- f'(r)\vol(S_r)\right)\\
      ={}& \frac{\vol(S_r)}{f(r)^2}\left(nf(r)\frac{1-n\sqrt{\kappa}r}{r}  - f'(r)\right)
  \end{align*}
  where the expression in the last line vanishes, since $f$ satisfies
  \begin{align*}
    f'(r) ={}& nr^{n-1}\exp(-\sqrt{\kappa}n^2r) - r^n\exp(-\sqrt{\kappa}n^2r)\sqrt{\kappa}n^2 \\
      ={}& nf(r)\frac{1 - n\sqrt{\kappa}r}{r}.
  \end{align*}
  This proves monotonicity of $\alpha$.
  The convergence follows from the fact that in normal coordinates the $1$-jets of $g$ and $g_\eucl$ agree at $0$.
  Therefore, $\vol(S_r)/(w_nr^n)$ converges to $1$ for $r\to0$ and hence
  \begin{align*}
    \alpha(r) = \frac{\vol(S_r)}{f(r)}= w_n\cdot \frac{\vol(S_r)}{w_nr^n}\exp(\sqrt{\kappa}n^2r)\too w_n
  \end{align*}
  for $r\to 0$.
  The monotonicity of $\alpha$ implies that $\alpha(r)\ge w_n$ for all $r>0$, and thus we have
  \[\vol(S_r) = \alpha(r)\cdot f(r) \ge w_n\cdot f(r)\qedhere\]
\end{proof}

\begin{lemma}\label{lem:hypersurfaces-in-covers}
  Let $d\le 11$ and let $W\colon\Sigma\leadsto \Sigma$ be a $d$-dimensional, orientable self-cobordism of a connected manifold $\Sigma$, equipped with a Riemannian metric $g$ that induces a smooth metric when gluing together the boundaries of $W$.
  Then, there exists an $m\in\bbN$ such that 
  \[m\cdot W\coloneqq \underbrace{W\cup_\Sigma\dots\cup_\Sigma W}_{m\text{-times}}\]
  equipped with the metric $g_m\coloneqq g\cup\dots\cup g$ contains a smooth, area-minimzing hypersurface $S$ in its interior which separates the two boundary components of $m\cdot W$.
\end{lemma}

\begin{proof}
  Let $\overline W$ (resp. $\overline W_m$) be the closed manifolds obtained by gluing the two boundary components of $W$ (resp. $m\cdot W$), see \cref{fig:self-cobordism-and-gluing}.
  We equip $\overline W$ (resp. $\overline W_m$) with the metric $\overline g$ (resp. $\overline g_m$) induced by $g$ (resp. $g_m$).
  Note, that the canonical map $p_m\colon \overline W_m\to\overline W$ is a covering and that $\overline g_m\coloneqq p_m^*\overline g$.
  \begin{figure}[ht]
    \scalebox{.8}{
      \begin{tikzpicture}
        \node at (-3,0){\includegraphics[width=.25\textwidth]{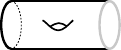}};
        \node at (-3,1.2) {$W$};
        \node at (-1.6,1.2) {$\Sigma$};
        
        \node at (3,0){\includegraphics[width=.45\textwidth]{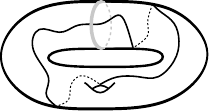}};
        \node at (1.5,2) {$\overline W$};
        \node at (3,2) {$\overline\Sigma$};
        \node (0) at (0,-2) {$\overline\gamma$};
        \draw[-stealth, bend left=30] (0) to (1.3,-.8);
      \end{tikzpicture}
    }
    \caption{The self-cobordism $W\colon \Sigma\leadsto\Sigma$ and the closed manifold $\overline W$ obtained by gluing the boundaries of $W$ containing $\overline\Sigma\cong \Sigma$ which intersects a loop generating a free cyclic subgroup of $\pi_1(W)$ transversely in a single point.}\label{fig:self-cobordism-and-gluing}
  \end{figure}
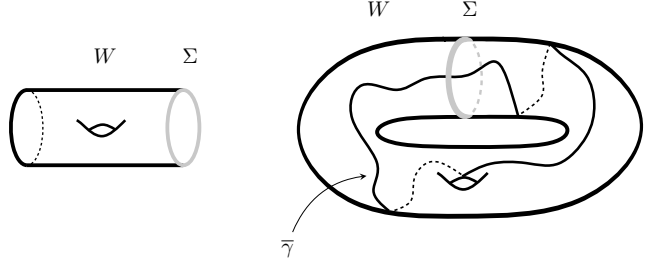
  The manifold $\overline W_m$ contains $m$ isometric copies of $\Sigma$ and $m\cdot W$ can be obtained by cutting open $\overline W_m$ along any of these copies.
  Any of these copies yields a non-trivial class in $\sigma\in\homology_{d-1}(\overline{W}_m;\bbZ)$, since it has intersection number $1$ with an appropriately chosen loop $\overline\gamma$ which generates an infinite cyclic subgroup of $\pi_1(\overline{W}_m)$.
  By \cite[Theorem 1.2]{Chodosh-Mantoulidis-Schulze-Wang_Generic-regularity-for-minimizing-hypersurfaces-in-dimension-11}, there exists for every $m\ge1$ a smooth, area-minimizing hypersurface $S_m\subset \overline{W}_m$ in the homology class $\sigma$ which thus satisfies 
  \begin{equation}\label{eq:volume-of-minimizer}
    \vol(S_m)\le\vol(\Sigma_m) = \vol(\Sigma)\quad\text{for all } m\ge1.
  \end{equation}
  If there exists some copy $\overline\Sigma$ of $\Sigma$ such that $\overline\Sigma\cap S_m=\emptyset$, then we can cut open $\overline{W}_m$ along $\overline\Sigma$, and we obtain a hypersurface $S$ in the interior of $m\cdot W$ which is homologous to either boundary component.
  Since $\Sigma$ was assumed to be connected, $S_m$ has to separate the boundary components of $m\cdot W$.

  \medskip

  For a contradiction, let us assume that for every $m\ge1$ the manifold $S_m$ intersects every copy of $\Sigma$ in $W_m$ and let $x_i\in S_m$ lie in the intersection with the $i$-th copy of $\Sigma$.
  Since $\overline W$ is compact, there exists a constant $\kappa>0$ (independent of $m$) such that $-\kappa\le \sec_{\overline g}\le \kappa$ and since $\overline g_m$ is a pullback of $\overline g$, it satisfies the same estimate.
  Similarly, the injectivity radius of $\overline g_m$ is bounded from below by the injectivity radius of $\overline g$ and the width of $g$.
  We define
  \[r\coloneqq\frac14\min\left({\rm injrad}(\overline{W}),\; \frac{1}{(d-1)^2\sqrt{\kappa}},\; \width({W})\right),\]
  which is independent of $m$ but only depends on $g$ and $d=\dim(W)$.
  We observe that the balls $B(x_i,r)$ are pairwise disjoint by the choice of $x_i$ and $r$.
  By the monotonicity formula (\cref{lem:monotonicity-for-bounded-curvature} and \eqref{eq:uniform-volume-estimate} for $n=d-1$) we get
  \[\vol\bigl(S_m\cap B(x_i,r)\bigr) \ge \exp(-1)w_{d-1}r^{d-1}\]
  for every $i\in\{1,\dots,m\}$.
  Therefore, we get
  \begin{align*}
    \vol(S_m)&\ge \sum_{i=1}^{n} \vol\bigl(S_m\cap B(x_i,r)\bigr)\\      
      &\ge m\cdot \exp(-1)w_{d-1}r^{d-1}
  \end{align*}
  and hence $\vol(S_m)\to\infty$ for $m\to\infty$.
  However, this is a contradiction to the inequality $\vol(S_m)\le \vol(\Sigma)$ from \eqref{eq:uniform-volume-estimate} for $m$ large enough.
  Thus, there exists an $m\ge1$ such that $S_m$ does not intersect some copy of $\Sigma$.
\end{proof}

\begin{proof}[Proof of \cref{lem:criterion-for-non-connected-boundaries}]
  If $\partial M$ is connected, then this is a weaker statement than \cref{lem:doubling-conjecture-criterion-one-component}, so we may assume that there are at least $2$ boundary components.
  Let $g_\double$ be a psc-metric on $\double M$ and let $\Sigma$ be a component of $\partial M$.
  We define $M_0\coloneqq M\cup_{\partial M\setminus\Sigma} M^\op$, which is a self-cobordism of $\Sigma$, see \cref{fig:double-at-hypersurface}.
  For $m\in\bbN$ we furthermore define:
  \[m\cdot M_0\coloneqq \underbrace{M_0\cup_\Sigma\dots\cup_\Sigma M_0}_{m\text{-times}}\]
  and we equip $M_0$ (resp. $m\cdot M_0$) with the psc-metric $g_0$ (resp. $m\cdot g_0\coloneqq g_0\cup\dots\cup g_0$) induced by $g_\double$.
  \begin{figure}[ht]
    \scalebox{.9}{
    \begin{tikzpicture}
      \node at (-3,0) {\includegraphics[height=10em]{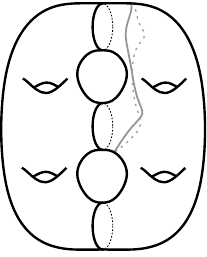}};
      \node at (-4,2){$\double M$};
      \node at (-2.5,2){\color{gray}$S$};
      \node at (-3,-2){$\Sigma$};
      
      \node at (3,0) {\includegraphics[height=8em]{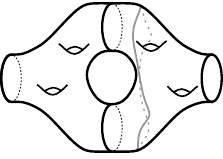}};
      \node at (1.5,1){$M_0$};
      \node at (1,-1){$\Sigma$};
      \node at (5,-1){$\Sigma$};
    \end{tikzpicture}}
    \caption{Cutting open $\double M$ along the boundary component $\Sigma$.
    Note, that unlike in this picture, $S$ might not be contained in $M_0$, but it will be contained in $m\cdot M_0$ for some possibly large $m\in\bbN$ by \cref{lem:hypersurfaces-in-covers}.}\label{fig:double-at-hypersurface}
  \end{figure}
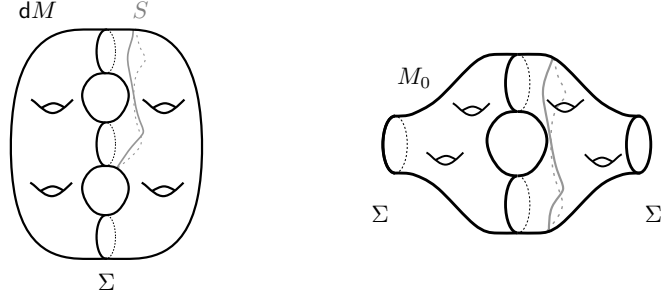

  \medskip
  
  An application of \cref{lem:hypersurfaces-in-covers} yields a smooth, area-minimizing hypersurface $S$ contained in the interior of $m\cdot M_0$ which separates the two  boundary components of $m\cdot M_0$.
  Thus, we get a decomposition
  \[m\cdot M_0 = M_1\cup_S M_2.\]
  We now consider the cobordism 
  \[M_2\cup_\Sigma M_0\cup_\Sigma M_1\colon S\leadsto S,\]
  see \cref{fig:double-cut-open-extended}.
  Since $S$ is minimal inside $m\cdot M_0$, we deduce that $m\cdot g_0\cup g_0\cup m\cdot g_0$ restricts to a psc-metric $g_1$ on $M_2\cup_\Sigma M_0\cup_\Sigma M_1$ with minimal boundary.
  By \cite[Corollary 4.3]{Baer-Hanke_boundary-conditions-for-scalar-curvature} the metric $g_1$ can be deformed in a neighborhood of the boundary so that the boundary becomes strictly mean convex.
  As the deformations in \cite{Baer-Hanke_boundary-conditions-for-scalar-curvature} only take place locally near the boundary, we observe that $g_1$ restricts to $g_0$ on $M_0$.
  \begin{figure}[ht]
    \scalebox{.9}{
    \begin{tikzpicture}
      \node at (0,0) {\includegraphics[width=.8\textwidth]{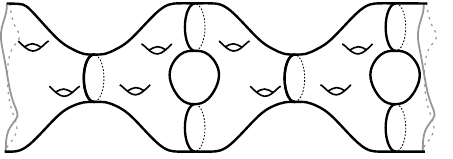}};
      \node at (-1,2.1){$M_0$};
      \node at (3,2.1){$M_1$};
      \node at (-4.5,2.1){$M_2$};
      \node at (4.5,-2){\color{gray}$S$};
      \node at (-5,-2){\color{gray}$S$};
      \node at (1.5,-1){$\Sigma$};
      \node at (-3,-1){$\Sigma$};
    \end{tikzpicture}}
    \caption{The self-cobordism $M_2\cup_\Sigma\cup M_0\cup_\Sigma M_1$ of $S$.}\label{fig:double-cut-open-extended}
  \end{figure}
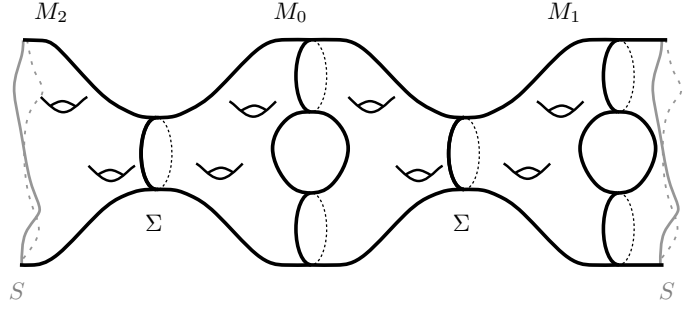
  
  By our assumption, the tangential $2$-type of $\Sigma$ extends to $M$ and hence also to $M_0$, $m\cdot M_0$, $M_1$ and $M_2$.
  Therefore, $M_1$ and $M_2$ are $\theta$-cobordant to manifolds $M_1'$ and $M_2'$ for which the inclusions $\Sigma\embeds M_1',M_2'$ are $2$-connected, see \cite[Appendix B]{hebestreitjoachim_twisted-spin-cobordism-and-positive-scalar-curvature} or \cite[Proposition 6.3]{ebertfrenck}.
  Employing \cref{lem:cobordism-principle-for-hpsc-bounding} we can extend the psc-metric $g_1$ on $M_2\cup_\Sigma M_0\cup_\Sigma M_1$ to a psc-metric $g_2$ with mean convex boundary on 
  \[(M_2')^\op\cup_S M_2\cup_\Sigma M_0\cup_\Sigma M_1\cup_S (M_1')^\op,\]
  see \cref{fig:double-at-hypersurfaces}.
  Since both $M_1\cup_S (M_1')^\op$ and $(M_2')^\op\cup_S M_2$ are both $\theta$-bordant to doubles relative to the boundary $\Sigma\amalg\Sigma$, \cref{prop:psc-from-doubles-to-cylinders} implies that there exists a psc-metric $g_3$ on $M_0$ with mean strictly convex boundary, which equals $g_0$ away from the boundary.
  By \cite[Corollary 4.3]{Baer-Hanke_boundary-conditions-for-scalar-curvature}, $g_3$ can be deformed in a neighborhood of the boundary so that it is doubling, that is, it induces a smooth metric on the double.
  In particular, both boundary components are minimal.
  \begin{figure}[ht]
    \scalebox{.9}{\begin{tikzpicture}
      \node at (0,0) {\includegraphics[width=1\textwidth]{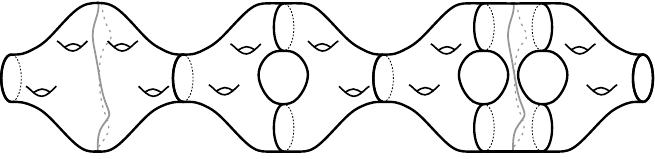}};
      \node at (-1,1.9){$M_0$};
      \node at (2,1.9){$M_1$};
      \node at (5.2,1.9){$(M_1')^\op$};
      \node at (-5.5,1.9){$(M_2')^\op$};
      \node at (-3.5,1.9){$M_2$};
      \node at (-4.5,-1.7){\color{gray}$S$};
      \node at (3.6,-1.7){\color{gray}$S$};
      \node at (1,-1){$\Sigma$};
      \node at (-2.8,-1){$\Sigma$};      
      \node(4) at (-4.5,-2.2) {$\underbrace{\qquad\qquad\qquad\quad\qquad}_{=(M_2')^\op\cup_{S}M_2\ \sim\  \Sigma\times[0,1]}$};
      \node(4) at (3.6,-2.2) {$\underbrace{\quad\qquad\qquad\qquad\qquad\quad\qquad\qquad}_{=M_1\cup_{S}(M_1')^\op\ \sim\  \Sigma\times[0,1]}$};
    \end{tikzpicture}}
    \caption{Attaching $(M_1')^\op$ and $(M_2')^\op$ to $M_2\cup_\Sigma M_0\cup_\Sigma M_1$.}\label{fig:double-at-hypersurfaces}
  \end{figure}
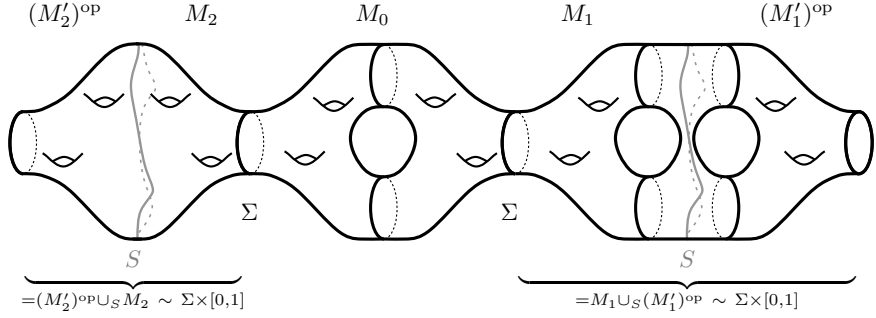
  
  \medskip

  It remains to adjust the boundary restrictions of the metric $g_3$, so that we obtain a smooth metric on $\double M$ upon gluing the boundaries. 
  We observe, that $M_0$ itself is also $\theta$-cobordant to $\Sigma\times[0,1]$ relative to the boundary and hence there exists a psc-metric $g_{\rm{cyl}}$ on $\Sigma\times[0,1]$ which is doubling and whose boundary restriction equals the one of $g_3$.
  Hence, we can flip it and glue the flipped metric $g_{\rm{cyl}}^\op$ on one side onto $g_3$ to obtain a metric $g_4$ on $M_0$ with equal boundary restrictions and both boundary components are doubling and in particular minimal, see \cref{fig:back-to-collars}.
  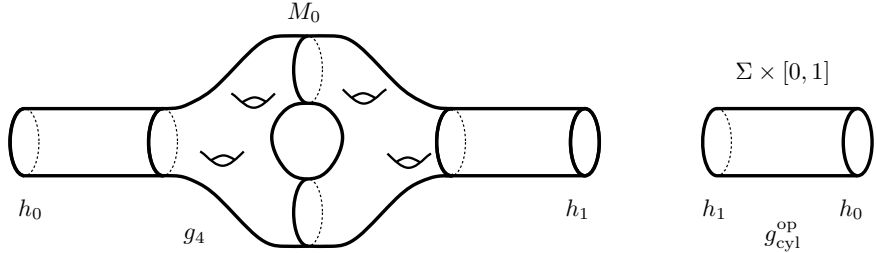
\begin{figure}[ht]
    \scalebox{.9}{\begin{tikzpicture}
      \node at (0,0) {\includegraphics[width=1\textwidth,trim=0 0 0 10,clip]{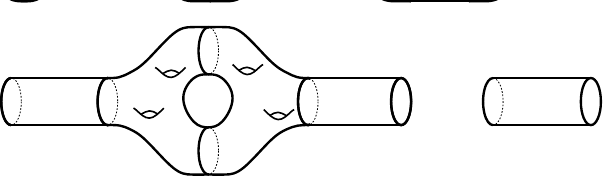}};
      \node at (-2,1.9){$M_0$};
      \node at (5,1){$\Sigma\times[0,1]$};
      \node at (-3.6,-1.4){$g_4$};
      \node at (-6,-1){$h_0$};
      \node at (2,-1){$h_1$};
      \node at (5,-1.4){$g_{\rm{cyl}}^\op$};      
      \node at (4,-1){$h_1$};
      \node at (6,-1){$h_0$};
    \end{tikzpicture}}
    \caption{Adjusting the boundary metrics on $M_0$ by gluing on the metric $g_{\rm cyl}^\op$ onto $g_3$.}\label{fig:back-to-collars}
  \end{figure}

  Gluing the boundary components of $(M_0,g_4)$, we obtain a smooth psc-metric $\overline g_4$ on $\double M$, such that 
  \begin{enumerate}
    \item $\overline g_4$ agrees with $g$ outside a neighborhood of $\Sigma\subset \double M$
    \item $\Sigma$ is minimal with respect to $\overline g_4$.
  \end{enumerate}
  In particular, we have $\overline g_4=g$ near all boundary components except $\Sigma$, which is now a minimal hypersurface in $\double M$.
  
  \medskip

  Therefore, we can perform the same procedure for every boundary component of $M$, and we obtain a psc-metric on $\double M$ such that $\partial M\subset \double M$ is a minimal hypersurface.
  Cutting open $\double M$ along the boundary, we obtain a psc-metric with minimal boundary on $M$, which, in particular has mean convex boundary, proving the \cref{doubling-conjecture} for $M$.
\end{proof}

\begin{remark}\label{rem:construction-is-local}
  The constructions in the proofs of \cref{lem:doubling-conjecture-criterion-one-component}, \cref{cor:doubling-for-extendable-tangential-two-type} and \cref{lem:criterion-for-non-connected-boundaries} are local in the following sense:
  If $g$ is a psc-metric on $\double M$, then the psc-metric on $M$ constructed in the respective proofs agrees with $g|_{M}$ away from the boundary $\partial M$.
\end{remark}

Next, we prove an inheritance lemma for the \cref{doubling-conjecture}, which is visualized in \cref{fig:inheritance}.

\begin{lemma}\label{lem:inheritance-for-doubling-conjecture}
  Let $M_1\colon\emptyset\leadsto\Sigma'$ and let $M_2\colon \Sigma'\leadsto \Sigma$ be a cobordism of dimension at least $5$ such that the following holds:
  \begin{enumerate}
    \item $M_1$ satisfies the \cref{doubling-conjecture}.
    \item The tangential $2$-type of $\Sigma'$ extends to $M_2$.
    \item For one component $\Sigma_1\subset\Sigma$ the inclusion $\Sigma_1\embeds M_2$ is $1$-connected.
    \item All components $\Sigma_i\subset \Sigma$ other than $\Sigma_1$ admit a psc-metric.
  \end{enumerate} 
  Then $M_1\cup_{\Sigma'} M_2$ also satisfies the \cref{doubling-conjecture}.
\end{lemma}

\begin{remark}
  In the case that $\Sigma'$ is connected, $(iv)$ is vacuously satisfied.
\end{remark}

\begin{figure}[ht]
  \scalebox{.8}{
    \begin{tikzpicture}
      \node at (0,0) {\includegraphics[width=.8\textwidth]{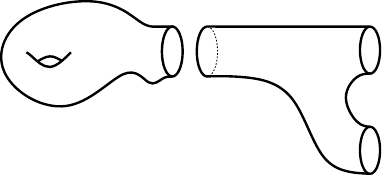}};
      \node at (-3,3) {$M_1$};
      \node at (-3,2.6){\scriptsize satisfies \cref{doubling-conjecture}};
      \node at (3,2.3) {$M_2$};
      \node at (5,2) {$\Sigma$};
      \node[anchor=west] at (5.1,1) {$\Sigma_1$};
      \node[anchor=west] at (5.1,-1.7) {$\Sigma_i$, $i\ge2$};
      \node at (5.9,-2.1){\scriptsize admits psc};
      \node at (4,1.3) {\scriptsize $1$-conn.};
      \draw[-stealth](.2,1) to (1.2,1);
      \node[fill=white] at (1.5,1.27) {\scriptsize $2$-type extends};  
      \draw[-stealth](4.5,1) to (3.5,1);
      \node at (0.5,2) {$\Sigma'$};
    \end{tikzpicture}
  }
  \caption{The situation from \cref{lem:inheritance-for-doubling-conjecture}}\label{fig:inheritance}
\end{figure}

\begin{proof}[Proof of \cref{lem:inheritance-for-doubling-conjecture}]
  Let us assume that the double $\double(M_1\cup M_2)$ of $M_1\cup M_2$ admits a metric $g_\double$ of positive scalar curvature.
  Since the tangential $2$-type $\theta$ of $\Sigma'$ extends to $M_2$, the manifold $M_2\cup_{\Sigma} M_2^\op$ is a $\theta$-cobordism from $\Sigma'$ to itself by (ii).
  By \cref{prop:psc-from-doubles-to-cylinders}, there exists a psc-metric on $\Sigma'\times[0,1]$ that agrees with $g_\double|_{M_2\cup_{\Sigma}M_2^\op}$ near the boundary and can hence be extended over $M_1$ and $M_1^\op$ to give a psc-metric on 
  \(\double M_1\cong M_1\cup_{\Sigma'\times\{0\}}\Sigma'\times[0,1]\cup_{\Sigma'\times\{1\}} M_1^\op,\)
  see \cref{fig:remove_m2}.
  Thus, $\double M_1$ admits a psc-metric and since $M_1$ was assumed to satisfy the \cref{doubling-conjecture}, $M_1$ admits a psc-metric with mean convex boundary.

  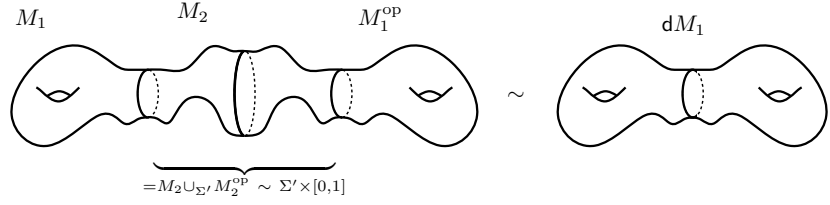
\begin{figure}[ht]
    \scalebox{.85}{\begin{tikzpicture}
      \node(0) at (0,0) {\includegraphics[width=\textwidth]{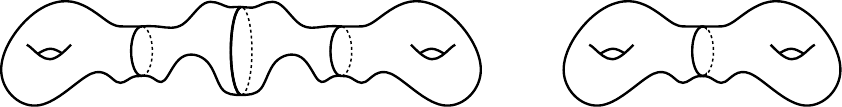}};
      \node(1) at (-6,1.2) {$M_1$};
      \node(2) at (-0.6,1.2) {$M_1^\op$};
      \node(3) at (-3.5,1.3) {$M_2$};
      \node(4) at (-2.7,-1.3) {$\underbrace{\quad\qquad\quad\qquad\qquad}_{=M_2\cup_{\Sigma'}M_2^\op\ \sim\  \Sigma'\times[0,1]}$};
      \node(5) at (4.1,1.1) {$\double M_1$};
      \node(6) at (1.5,0) {$\sim$};
    \end{tikzpicture}}
    \caption{The double of $M_1\cup M_2$ is cobordant to the double of $M_1$.}\label{fig:remove_m2}
  \end{figure}
  
  Next, we choose for every $i\ge2$ a metric $h_i$ on $\Sigma_i$ with $\scal(h_i)\ge1$, and we consider the metric $g_i = \dt^2 + f(t)^2 h_i$ for $f(t) = 1+\tfrac{t^2}{8(d-1)}$ on $\Sigma_i\times[-1,1]$.
  It has scalar curvature given by
  \begin{align*}
    \scal(g_i) 
      =&{}\frac{1}{f^2}\left(\scal(h_i)-(d-1)(d-2)f'^2-2(d-1)ff''\right)
      \ge\frac{1}{f^2}\left(\frac12-\frac{t^2}{16}\right),
  \end{align*}
  which is positive since $t^2\le 1$.
  The mean curvature of the slices $\Sigma_i\times\{\pm1\}$ with respect to the outward pointing normal vector fields is given by
  \begin{align*}
    H(g_i)|_{\Sigma_i\times\{\pm1\}} = \pm(d-1) \frac{f'(t)}{f(t)} 
    = \frac{2(d-1)}{8d-7} > 0.
  \end{align*}  
  So, the metrics $g_i$ have positive scalar curvature and strictly mean convex boundary.
  If we define $\widetilde\Sigma\coloneqq\Sigma\setminus\Sigma_1$, we get a cobordism 
  \[W\coloneqq M_1\amalg\left(\widetilde\Sigma \times[-1,1]\right)\colon\widetilde\Sigma\leadsto \Sigma'\amalg\widetilde{\Sigma},\]
  which admits a psc-metric with strictly mean convex boundary, see \cref{fig:weak-inheritance}.
  $M_2$ can now be interpreted as a cobordism
  \[M_2\colon \Sigma\amalg \widetilde \Sigma\leadsto \Sigma_1,\]
  and, since the inclusion of $\Sigma_1\embeds M_2$ is $1$-connected, \cref{thm:wall-geometric-connectivity} and \cref{lem:cobordism-principle-for-hpsc-bounding} imply the existence of a psc-metric on $W\cup_{\Sigma'\amalg\widetilde{\Sigma}} M_2$ with mean convex boundary, see \cref{fig:weak-inheritance}.
  The proof is finished by the observation that $W\cup_{\Sigma'\amalg\widetilde{\Sigma}} M_2$ is diffeomorphic to $M_1\cup_{\Sigma'} M_2$.
  \begin{figure}[ht]
    \scalebox{0.8}{
    \begin{tikzpicture}
      \node at (-3.7,0) {\includegraphics[width=.5\textwidth]{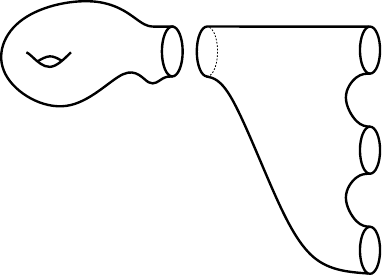}};
      \node at (-5.2,2.5){$M_1$};
      \node at (-2,2.4){$M_2$};
      \node at (-3.9,2.3){$\Sigma'$};
      \node at (-.5,2.4){$\Sigma$};
      \node at (-0.2,1.2){$\Sigma_1$};
      \node at (-0.2,-0.4){$\Sigma_2$};
      \node at (-0.2,-2){$\Sigma_3$};
      
      \node at (3.7,0) {\includegraphics[width=.5\textwidth]{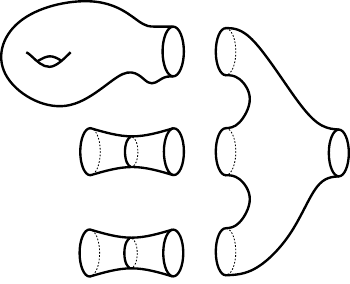}};
      \node at (2,2.8){$M_1$};
      \node at (6,1.9){$M_2$};
      \node at (2.3,-3.2){$\underbrace{\qquad\qquad\qquad\qquad}$};
      \node at (2.3,-3.7) {$=W$};
      \node at (6.8,1){$\Sigma_1$};
      \node at (4.1,1.3){$\Sigma'$};
      \node at (4.2,-0.5){$\Sigma_2$};
      \node at (4.2,-2.3){$\Sigma_3$};
      \node at (2.9,-.9){$\Sigma_2\times[-1,1]$};
      \node at (2.9,-2.7){$\Sigma_3\times[-1,1]$};
    \end{tikzpicture}}
    \caption{Reinterpreting the cobordism $M_2\colon \widetilde{\Sigma}\amalg\Sigma\leadsto \Sigma_1'$.}\label{fig:weak-inheritance}
  \end{figure}
\end{proof}

\section{Proof of \texorpdfstring{\cref{main:doubling-conjecture-spin-and-nonspin} and \cref{main:doubling-conjecture-almost-spin}}{the doubling conjecture}}\label{sec:proof-of-main}

\begin{proof}[Proof of \cref{main:doubling-conjecture-spin-and-nonspin}]
  We start by considering the case that $M$ is spin.
  Let $N_i\subset\partial M$ be a component and let $N_i\times[0,1]\subset M$ be a collar of $N_i$, where $N_i\times\{1\}$ corresponds to the boundary.
  We claim that $\ker\bigl((\iota_i)_\ast\colon\pi_1(N_i)\to\pi_1(M)\bigr)$ is normally finitely generated:
  By the split-injectivity assumption, $\pi_1(M) \cong \im(\iota_i)_\ast\times R$.
  The group $R$ is finitely generated since the projection $\pi_1(M)\to R$ is surjective and $\pi_1(M)$ is finitely generated.
  Let $\pi_1(M) = \scpr{a_1,\dots,a_n\ |\ r_1,\dots,r_m}$ be a presentation of $\pi_1(M)$ and let $b_1,\dots,b_l$ be a set of generators of $R$.
  Then $\scpr{a_1,\dots,a_n\ |\ r_1,\dots,r_m,b_1,\dots,b_l}$ is a presentation of $\im(\iota_i)_\ast$.
  Hence, $(\iota_i)_\ast\colon \pi_1(N_i)\to \im(\iota_i)_\ast$ is a surjective homomorphism from a finitely generated group to a finitely presented group and hence, $\ker(\iota_i)_\ast$ is normally finitely generated by \cite[Lemma 3.2]{Schick-Zenobi_Positive-scalar-curvature-due-to-the-cokernel-of-the-classifying-map}.

  \medskip

  We choose embedded loops $\alpha_1,\dots,\alpha_m\colon S^1\embeds N_i\times\{0\}$ whose representatives generate $K_i\coloneqq\ker\bigl(\pi_1(N_i\times\{0\})\to\pi_1(M)\bigr)$ normally.
  Since $M$ is orientable, all these loops have trivial normal bundle.
  Furthermore, there exist disjoint embedded $2$-handles $H_j\colon D^2\times D^{d-2}\embeds M\setminus \bigl(N_i\times(0,1]\bigr)$ such that $H_j|_{D^2\times\{0\}}$ extends $\alpha_j$ for each $j$, because $\dim(M)\ge5$ and the generators $[\alpha_j]$ are in the kernel of $(\iota_i)_\ast$.
  After smoothing corners, we obtain a codimension $0$ submanifold $W$ of $M$ given by
  \[W\coloneqq N_i\times[0,1]\cup\bigcup_{j=1}^m H_j(D^2\times D^{d-2})\colon N_i\leadsto \Sigma_i,\]
  where $\Sigma_i$ is obtained from $N_i$ by performing surgeries along  $H_j|_{S^1\times D^{d-2}}$ on $N_i$, see \cref{fig:2-handle}.
  The manifold $W$ consists of $2$-handles attached to $N_i$ and therefore the inclusion $N_i\embeds W$ is bijective on path components and surjective on $\pi_1$, hence it is $1$-connected.
  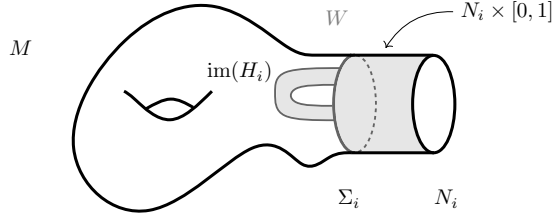
\begin{figure}[ht]
    \scalebox{.8}{
    \begin{tikzpicture}
      \node at (0,0) {\includegraphics[width=0.5\textwidth]{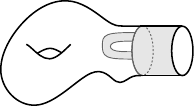}};
      \node at (-4,1) {$M$};
      \node at (1.2,1.5) {\textcolor{gray}{$W$}};
      \node(0) at (4,1.6) {$N_i\times[0,1]$};
      \node at (-0.4,0.6) {$\im(H_i)$};
      \node at (1.4,-1.5) {$\Sigma_i$};
      \node at (3,-1.5) {$N_i$};
      \draw[->, bend right=30] (0.west) to (2,1);
    \end{tikzpicture}}
    \caption{The codimension $0$ submanifold $W$ of $M$ consisting of $2$-handles (gray area) attached to $N_i$.}\label{fig:2-handle}
  \end{figure}
  Furthermore, $W$ is diffeomorphic to the manifold obtained by attaching $(d-2)$-handles to $\Sigma_i\times[0,1]$.
  This implies that the inclusion $\Sigma_i\embeds W$ is $2$-connected and therefore, the tangential $2$-type of $\Sigma$ extends to $M$ by \cref{rem:on-extendable-tangential-2-type}.
  By \cref{lem:inheritance-for-doubling-conjecture}, it thus suffices to prove the \cref{doubling-conjecture} for $M\setminus W$.

  \medskip

  Note, that $\pi_1(\Sigma_i)\cong \pi_1(N_i)/K$ and therefore, the inclusion $\Sigma_i\embeds M\setminus W$ induces a split-injection on $\pi_1$.
  By \cref{prop:tangential-2-type-extends}, the tangential $2$-type of the $\Sigma_i$ extends to $M\setminus W$.
  By performing the same procedure for every boundary component, we may reduce to the case that the tangential $2$-type of every boundary component of $M\setminus W$ extends to $M\setminus W$.

  \medskip

  An application of \cref{lem:doubling-conjecture-criterion-one-component} (in case $\partial M$ is path-connected) or \cref{lem:criterion-for-non-connected-boundaries} (if $\dim(M)\le11$) finishes the proof in the respective cases.
  If $\partial M$ is not path-connected and $\dim(M)\ge12$, we observe that the extra assumption from \cref{main:doubling-conjecture-spin-and-nonspin} implies that every component $\Sigma_i$ of $\partial(M\setminus W)$ is simply connected and admits a psc-metric by \cite[Theorem A]{stolz_simplyconnected}, except for possibly one, say $\Sigma_1$.
  For every $i\ge2$ we choose some thickened path connecting $\Sigma_i$ to $\Sigma_1$ and consider the codimension $0$ submanifold $W'\subset M\setminus W$ given by the union of a collar of the boundary and the chosen paths.
  We note that $W'$ is a cobordism from $\Sigma' \cong \Sigma_1\#\dots\#\Sigma_k$ to $\Sigma$ and the inclusions $\Sigma'\embeds W'$ and $\Sigma_1\embeds W'$ both induce isomorphisms on fundamental groups.
  Hence, $\Sigma_1\embeds W'$ is $1$-connected and the tangential $2$-type of $\Sigma'$ extends to $W'$. 
  \cref{lem:inheritance-for-doubling-conjecture} implies that it is sufficient to prove the \cref{doubling-conjecture} for $M\setminus(W\cup W')$, which has connected boundary $\Sigma'$.
  Thus, the spin-case is finished by an application of \cref{lem:doubling-conjecture-criterion-one-component}.

  \medskip

  If $M$ is totally nonspin, we first reduce to the case that every component of $\partial M$ is totally nonspin, too.
  By \cref{prop:criterion-for-spin-and-nonspin}, there exists an embedded $2$-sphere in the interior of $M$ with non-trivial normal bundle and we take a tubular neighborhood $T$ of this $2$-sphere.
  If $N_i$ is a component of $\partial M$, we perform the internal boundary connected sum $W'\coloneqq T\natural (N_i\times[0,1])$, that is, we take the union of $T$, a collar $N_i\times[0,1]$ of $N_i$ and a thickened path inside $M$ connecting $\partial T$ and $N_i\times\{0\}$, see \cref{fig:totally-nonspin}.

  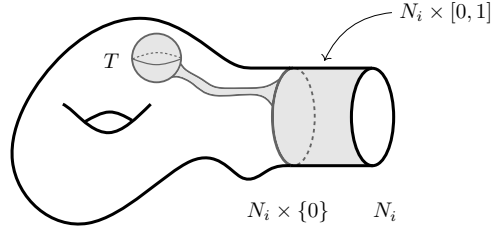
\begin{figure}[ht]
    \scalebox{.8}{
    \begin{tikzpicture}
      \node at (0,0) {\includegraphics[width=0.5\textwidth]{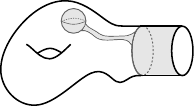}};
      \node(0) at (4,1.8) {$N_i\times[0,1]$};
      \node at (-1.5,1) {$T$};
      \node at (1.4,-1.5) {$N_i\times\{0\}$};
      \node at (3,-1.5) {$N_i$};
      \draw[->, bend right=30] (0.west) to (2,1);
    \end{tikzpicture}}
    \caption{The cobordism $W'$ from $\Sigma$ to a totally nonspin-manifold $\Sigma'$.}\label{fig:totally-nonspin}
  \end{figure}
  
  \medskip

  Note, that $W'$ is a cobordism $\Sigma_i\leadsto N_i$ for $\Sigma_i\coloneqq N_i\# \partial T$.
  Since $\partial T$ is simply connected and totally nonspin, the manifold $\Sigma_i$ also is totally nonspin and $\pi_1(\Sigma_i)\cong\pi_1(W')$.
  Hence, the tangential $2$-type of $\Sigma_i$ extends to $W'$ by \cref{prop:tangential-2-type-extends}.
  Furthermore, the inclusion $N_i\embeds W'$ is $1$-connected.
  By \cref{lem:inheritance-for-doubling-conjecture} it thus suffices to show that $M\setminus W'$ satisfies the \cref{doubling-conjecture}.
  Iterating this construction we may assume that every boundary component is totally nonspin.

  \medskip

  By the same argument as above, we may further reduce to the case that the inclusion of every boundary component induces a split-injection on $\pi_1$.
  By \cref{prop:tangential-2-type-extends}, the tangential $2$-type of every boundary component extends to $M$.
  The \cref{doubling-conjecture} follows from \cref{lem:doubling-conjecture-criterion-one-component} (if $\partial M$ is connected), \cref{lem:criterion-for-non-connected-boundaries} (if $\dim(M)\le 11$) or from the combination of \cref{lem:inheritance-for-doubling-conjecture} and \cref{lem:doubling-conjecture-criterion-one-component} as before.
\end{proof}

\noindent The proof for the almost spin case is very similar.

\begin{proof}[Proof of \cref{main:doubling-conjecture-almost-spin}]
  As in the proof of \cref{main:doubling-conjecture-spin-and-nonspin}, we may assume that the inclusion of every boundary component induces a split-injection on fundamental groups.
  By \cref{prop:tangential-2-type-extends}, the tangential $2$-type of every component of $\partial M$ extends to $M$.
  Hence, an application of \cref{lem:doubling-conjecture-criterion-one-component} or \cref{lem:criterion-for-non-connected-boundaries} finishes the proof in this case.
\end{proof}

Let us now explain how to adjust our techniques to prove the $4$-dimensional result from \cref{main:dimension-4}.

\begin{proof}[Proof of \cref{main:dimension-4}]
  For $(i)$ let $M$ be a $4$-manifold whose boundary consists of a disjoint union of $3$-spheres and let $g_\double$ be a psc-metric on $\double M$.
  Fix one boundary component $S\subset\partial M$, a collar $S\times[0,1]\subset M$, and choose disjoint embedded paths $\gamma_i$ connecting $S$ to the other boundary components.
  We perform internal boundary connected sums along $\gamma_i$ to connect the respective collars.
  This yields a codimension $0$ submanifold $W\subset M$, whose incoming boundary is given by the connected sum of the components of $\partial M$, which is diffeomorphic to the $3$-sphere, while the outgoing boundary is given by $\partial M$.
  We observe that $\double M$ contains $\double W$ as a submanifold, and we can perform surgery along the doubles $\double\gamma_i$  of the connecting paths $\gamma_i$ to turn this into a cylinder.
  Since these are surgeries of codimension $3$, the resulting manifold carries a psc-metric.
  Hence, we obtain a psc-metric on $M\setminus W\cup (S^3\times[0,1])\cup (M\setminus W)^\op \cong \double (M\setminus W)$.
  This implies that $M\setminus W$ admits a psc-metric with mean convex boundary by assumption.

  \medskip

  As in the proof of \cref{lem:inheritance-for-doubling-conjecture} (see \cref{fig:weak-inheritance}), there exists a psc-metric $g'$ on $(M\setminus W)\amalg \bigl((\partial M\setminus S)\times[-1,1]\bigr)$ with mean convex boundary.
  Furthermore, the cobordism $W$ only consists of $0$-handles.
  Therefore, by \cref{lem:cobordism-principle-for-hpsc-bounding} we can extend $g'$ over $W$ to a psc-metric on $M$ with mean convex boundary.

  \medskip

  In order to prove \cref{main:dimension-4} $(ii)$, we first note that $\pm K3$ admits a handle decomposition with one $0$- and $4$-handle and without $1$- and $3$-handles, see \cite[\S 2]{Harer-Kas-Kirby_Handlebody-decompositions-of-complex-surfaces}.
  Therefore, the same is true for the connected sum of any number of copies of these and by \cref{lem:cobordism-principle-for-hpsc-bounding}, $(m\cdot K3)\setminus D^4$ admits positive scalar curvature with positive mean curvature at the boundary for every $m\in\bbZ$.
  By \cref{thm:cobordism-theorem-for-psc}, the connected sum of manifolds carrying a psc-metric again admits a psc-metric, which completes the proof.

  \medskip

  For the proof of $(iii)$ we use \cite[Theorem C]{kreck_surgery-and-duality}, which states the following:
  If $M_1$ and $M_2$ are closed $4$-manifolds with the same Euler-characteristic and with the same (stable) tangential $2$-type $\theta$\footnote{
    Kreck's result has an assumption on the stable normal $2$-type of $M_1$ and $M_2$.
    Since stable normal structures are equivalent to stable tangential structures, this is equivalent to working with the tangential $2$-type as defined in \cref{def:tangential-2-types}.
  }, which are $\theta$-cobordant, then $k(S^2\times S^2)\#M_1$ and $k(S^2\times S^2)\#M_2$ are diffeomorphic for some $k\ge0$.

  \medskip
  
  If $M$ is spin and $\pi_1(M)$ is isomorphic to the free group $F_n$ for some $n\ge0$, then the tangential $2$-type is given by $\bspin\times BF_n \simeq \bspin\times (S_1\vee\dots\vee S^1)$.
  The inclusion homomorphism $\Omega^{\Spin}_4\embeds\Omega^{\Spin}_4(BF_n)$ of the respective spin cobordism groups is an isomorphism by the Atiyah--Hirzebruch spectral sequence, because $F_n$ has homological dimension $1$ (or $0$ if $n=0$).
  Therefore, the manifold $M$ is $\theta$-cobordant to $m\cdot K3\# n\cdot(S^1\times S^3)$, where $m=-\sign(M)/16$.
  By Kreck's result, there exist $k,\tilde k\ge0$ such that $k(S^2\times S^2)\#M$ is diffeomorphic to $\tilde k(S^2\times S^2)\#m\cdot K3\# n(S^1\times S^3)$ and the latter admits a psc-metric with mean convex boundary.
  Hence, we obtain a psc-metric with mean convex boundary on $k(S^2\times S^2)\#M$.

  \medskip

  If $M$ is totally nonspin and $\pi_1(M)$ has homological dimension at most $3$, the Atiyah--Hirzebruch spectral sequence reveals that the inclusion $\Omega_4\embeds\Omega_4(B\pi_1(M))$ is an isomorphism, where $\Omega_4$ denotes the oriented $4$-th oriented cobordism group.
  Taking the connected sum of one copy of $S^1\times S^3$ for every generator of $\pi_1(M)$ and performing $1$-surgeries on the resulting manifold to realize the (finite number of) relations of $\pi_1(M)$, we obtain a closed, oriented $4$-manifold $M_0$ with $\pi_1(M_0)\cong \pi_1(M)$, which admits a psc-metric.
  For appropriately chosen $n,m\in\bbN$, the manifold $M$ is $\theta$-cobordant to $M_0\#m\cp2\#n\bcp2$ which has the same tangential $2$-type as $M$.
  As above, there exist $k,\tilde k\ge0$ such that $k(S^2\times S^2)\#M$ is diffeomorphic to $\tilde k(S^2\times S^2)\#M_0\#m\cp2\#n\bcp2$, and the proof is finished as in the spin case.
\end{proof}

\section{Minimizing separating hypersurfaces}\label{sec:separating-hypersurfaces}
In this section we investigate \cref{que:minimal-hypersurface}.
Before proving \cref{main:minimal-surfaces}, let us discuss two examples which provide negative answers to \cref{que:minimal-hypersurface}.
These illustrate that there cannot be an affirmative answer without additional assumptions.

\begin{example}\label{ex:not-stable-minimal}
  There exists a closed manifold $M$ admitting positive scalar curvature and a closed hypersurface $\Sigma\subset M$ such that $\Sigma$ is not stable minimal with respect to any psc-metric on $M$:

  Let $A$ be a manifold such that $\Sigma\coloneqq A\times S^1$ does not admit positive scalar curvature, for example $A$ could be a product of $K3$-surfaces and tori.
  Let $M\coloneqq A\times S^2$ which contains $\Sigma$ as the product of $A$ and the equator in $S^2$.
  Then $M$ admits positive scalar curvature but $\Sigma$ cannot be stable minimal with respect to any psc-metric on $M$, because the induced metric on $\Sigma$ would then be conformal to a psc-metric, see \cite[Proof of Theorem 1]{schoenyau_classical}.
\end{example}

\begin{example}\label{ex:not-minimal}
  There exists a closed manifold $M$ admitting positive scalar curvature and a closed hypersurface $\Sigma\subset M$ such that $\Sigma$ is not minimal with respect to any psc-metric on $M$:
  
  Let $\beta = K3\times\dots\times K3$, which is a simply connected, closed spin manifold with non-vanishing $\ahat$-genus and consider for $n\ge1$
  \begin{align*}
    M_1&\coloneqq (T^n\setminus D^n)\times \beta\\
    M_2&\coloneqq M_1 \# X,
  \end{align*}
  for $X$ a closed, simply connected nonspin manifold of dimension $(n+\dim\beta)$, for example, $\cp2\times S^{n+\dim(\beta)-4}$ if $\dim(\beta)\ge5$ or $n\ge2$.
  We define 
  \[\Sigma\coloneqq S^{n-1}\times\beta=\partial M_1=\partial M_2\quad\text{and}\quad M\coloneqq M_1\cup_\Sigma M_2.\]
  The manifold $M$ is totally nonspin and its fundamental group is given by $\bbZ^n\ast\bbZ^n$, which has homological dimension $n<\dim(M)$.
  Therefore, $M$ admits a metric of positive scalar curvature by \cite[Theorem 1.2]{fuehring_bordism-and-projective-space-bundles}, see also \cite[Theorem A.1]{GromovHanke_torsion-obstructions-to-positive-scalar-curvature}.
  However, the hypersurface $\Sigma$ cannot be minimal for any psc-metric on $M$.
  Otherwise, we could cut open $M$ along $\Sigma$, and we would obtain a psc-metric on $M_1$ with minimal boundary.
  By \cite[Corollary 4.3]{Baer-Hanke_boundary-conditions-for-scalar-curvature}, this psc-metric could be deformed into one which is doubling, that is it would induce a psc-metric on the double of $M_1$.
  But $\double M_1 = (T^n\#-T^n)\times\beta$ is the product of an enlargeable spin manifold and a spin manifold with non-vanishing $\ahat$-genus and hence cannot admit positive scalar curvature by \cite[Theorem 2.19]{Baer-Hanke_boundary-conditions-for-scalar-curvature}.
\end{example}
  
Both of these example crucially rely on the fact that the tangential $2$-type of the $\Sigma$ does not extend to $M$.
In \cref{ex:not-minimal}, we furthermore have different tangential $2$-types for $M_2$ and $M_1^\op$, and in particular they are not diffeomorphic.
Therefore, this construction does not yield a counterexample to the \cref{doubling-conjecture}.

\medskip  

Let us now turn to the proof of \cref{main:minimal-surfaces}.
Again, this follows from a more general criterion involving tangential structures.

\begin{lemma}\label{thm:criterion-for-minimal-surface}
  Let $M$ be a closed oriented manifold, let $g$ be a psc-metric on $M$ and let $\Sigma\subset M$ be a two-sided connected hypersurface.
  Assume that the tangential $2$-type of $\Sigma$ extends to $M$. 
  If $\Sigma$ is non-separating, assume further that $\dim(M)\le 11$.
  Then: 
  \begin{enumerate}
    \item If $\dim(M)\ge5$, then there exists a psc-metric $\widetilde{g}$ such that $\Sigma$ is minimal with respect to $\widetilde{g}$. The metric $\widetilde{g}$ can be chosen to agree with $g$ outside a tubular neighborhood of $\Sigma$.
    \item If $\dim(M)\ge6$, then there exists a psc-metric $\widetilde{g}$ which is of the form $g_\Sigma + \dt^2$ near $\Sigma$. 
    If $\Sigma$ is separating, the metric $\widetilde{g}$ can be chosen to agree with $g$ outside a tubular neighborhood of $\Sigma$.
  \end{enumerate}
\end{lemma}

\begin{remark}\label{rem:dim-restriction-is-necessary}
  The dimension restriction in the second part of \cref{thm:criterion-for-minimal-surface} is necessary as shown by the following example: 

  By \cite[Theorem 3]{hanke-kotschick-wehrheim_dissolving-four-manifolds-and-positive-scalar-curvature}, there exists a simply connected, spin $4$-manifold $\Sigma$ with vanishing signature, non-vanishing Seiberg--Witten invariant and arbitrarily large $b^+$.
  Hence, $\Sigma$ does not admit a metric of positive scalar curvature and is oriented nullbordant.The natural map $\Omega_4^\Spin\to \Omega_4$ from the fourth spin cobordism group to the fourth oriented cobordism group is injective, so there is a $5$-dimensional spin manifold $M$ with boundary $\Sigma$.
  After performing surgeries on the interior of $M$, we may assume that $M$ is also simply connected.
  The double of $M$ is a simply connected spin $5$-manifold, hence, the tangential $2$-type of $\Sigma$ extends to the double $\double M$.
  Furthermore, $M$ admits a psc-metric with strictly mean convex boundary by \cref{cor:boundaries-with-hpsc-metrics}, which can be deformed to be doubling by \cite[Corollary 4.3]{Baer-Hanke_boundary-conditions-for-scalar-curvature}.
  Therefore, $\Sigma$ is minimal with respect to the doubled metric.
  However, if there was a psc-metric on $\double M$ for which $\Sigma$ is a stable minimal hypersurface, the induced metric on $\Sigma$ would be conformal to a psc-metric, see \cite[Proof of Theorem 1]{schoenyau_classical}.
\end{remark}

\begin{proof}[Proof of \cref{thm:criterion-for-minimal-surface}]
  We first handle the case that $\Sigma$ is separating, say $M = M_1\cup_\Sigma M_2$.
  Consider the disjoint union $M\amalg M^\op$ and note, that $M_2\amalg M_2^\op\subset M\amalg M^\op$.
  An application of \cref{prop:psc-from-doubles-to-cylinders} yields a psc-metric on $\double M_1$, see \cref{fig:remove_dm2}.
  \begin{figure}[ht]
    \scalebox{.9}{
    \begin{tikzpicture}
      \node at (0,0) {\includegraphics[width=1\textwidth]{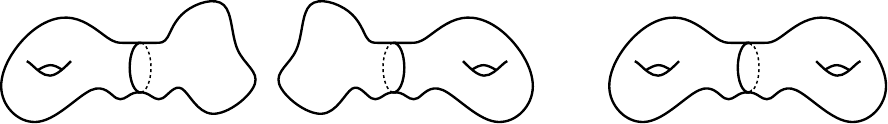}};
      \node at (-6,1.2) {$M$};
      \node at (-1,1.2) {$M^\op$};
      \node(0) at (4.4,1.2) {$\double M_1$};
      \node at (-5.5,-1.1) {$M_1$};
      \node at (-3.6,-1.1) {$M_2$};
      \node at (-4.5,0.9) {$\Sigma$};
      \node at (1.85,0) {$\sim$};
      \node at (-2.5,-1.6) {$\underbrace{\quad\qquad\quad\qquad\qquad\qquad}_{=\  M_2\,\amalg\, M_2^\op\ \sim\  \Sigma\times[0,1]}$};
    \end{tikzpicture}}
    \caption{Obtaining $\double M_1$ from the double of the closed manifold $M$.}\label{fig:remove_dm2}
  \end{figure}
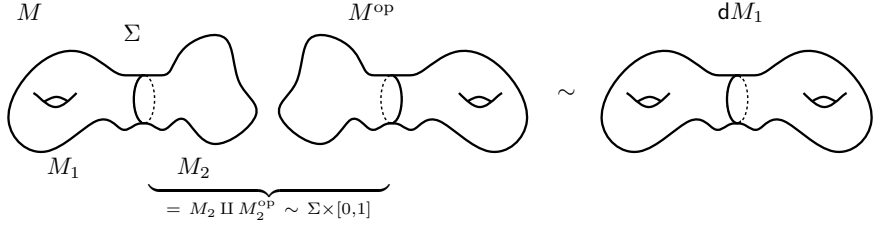

  \medskip

  Since the tangential $2$-type of $\Sigma$ extends to $M$, it also extends to $M_1$ and by \cref{lem:doubling-conjecture-criterion-one-component}, there is a psc-metric on $M_1$ with mean convex boundary.
  By \cite[Corollary 4.3]{Baer-Hanke_boundary-conditions-for-scalar-curvature}, this can be deformed in a neighborhood of the boundary to a psc-metric $g_1$, which is doubling.
  In particular, $\Sigma$ is minimal with respect to $g_1\cup g_1^{\op}$.

  \medskip

  Next, consider the manifold $\double M_1\amalg M$ which carries the psc-metric $(g_1\cup g_1^\op)\amalg g$.
  Again, by \cref{prop:psc-from-doubles-to-cylinders}, we get a psc-metric $\widetilde g$ on $M\cong M_1\cup_{\Sigma}\Sigma\times[0,1]\cup_\Sigma\cup M_2\sim \double M_1\amalg M$ such that the restriction of $\widetilde g$ to $M_1$ equals $g_1$, see \cref{fig:back_to_m}.
  Therefore, $\widetilde g$ satisfies all the properties claimed above.
  \begin{figure}[ht]
    \scalebox{.9}{\begin{tikzpicture}
      \node at (0,0) {\includegraphics[width=1\textwidth]{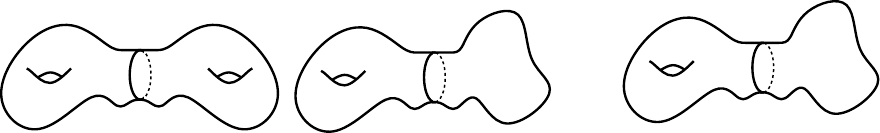}};
      \node at (4,1.1) {$M$};
      \node at (-1,1) {$M$};
      \node at (-5.5,1) {$M_1$};
      \node at (-3,1) {$M_1^\op$};
      \node at (-4.4,-0.9) {\footnotesize $\Sigma$};
      \node at (-5,-0.1) {$g_1$};
      \node at (4.1,0) {$g_1$};
      \node at (-4.2,0.6) {$g_1\cup g_1^\op$};
      \node at (0,0.6) {$g$};
      \node at (4.6,0.6) {$\tilde g$};
      \node at (2,0) {$\sim$};
      \node at (-1.4,-1.8) {$\underbrace{\quad\qquad\qquad\qquad\qquad\qquad\qquad\qquad\qquad}_{=\  M_1^\op\,\amalg\, M_1\,\cup_\Sigma\, M_2\ \sim\  \Sigma\times[0,1]\,\cup_\Sigma\, M_2\ \cong\  M_2}$};
    \end{tikzpicture}}
    \caption{Constructing the required psc-metric $\tilde g$ on $M$}\label{fig:back_to_m}
  \end{figure}
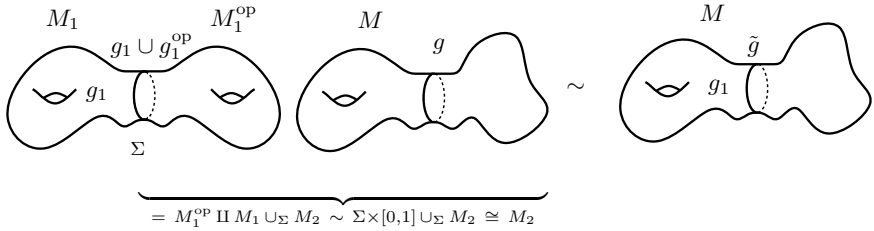

  \medskip

  If $\dim M\ge 6$, we can use \cref{cor:doubling-for-extendable-tangential-two-type} instead of \cref{lem:doubling-conjecture-criterion-one-component} to obtain a psc-metric $g_1$ on $M_1$ which is of product form near the boundary.
  Performing the second step as before, we obtain a psc-metric $\widetilde{g}$ which is of product form in a neighborhood of $\Sigma$.

  \medskip

  In the case that $\Sigma$ is non-separating, cutting $M$ open along $\Sigma$ yields a self-cobordism $M_0\colon \Sigma\leadsto \Sigma$.
  By \cref{lem:hypersurfaces-in-covers}, there exists a separating, area-minimizing hypersurface $S\subset m\cdot M_0$ for some $m\ge1$.
  As in the proof of \cref{lem:criterion-for-non-connected-boundaries}, we have a decomposition
  \[m\cdot M_0 = M_1\cup_S M_2,\]
  and we consider the cobordism 
  \[(M_2')^\op\cup_S M_2\cup_\Sigma M_0\cup M_1\cup_S(M_1')^\op\colon \Sigma\leadsto \Sigma,\]
  see \cref{fig:double-at-hypersurfaces}.
  By the same argument as in the proof of \cref{lem:criterion-for-non-connected-boundaries}, this cobordism admits a psc-metric with minimal boundary.

  \medskip

  Since $(M_2')^\op\cup_S M_2$ and $M_1\cup_S(M_1')^\op$ are both $\theta$-cobordant to cylinders over $\Sigma$, we obtain a psc-metric on $M_0$ with stably minimal boundary and whose boundary restrictions can be made to agree, see the argument around \cref{fig:back-to-collars}.
  This can be turned into a doubling psc-metric by an application of \cite[Corollary 4.3]{Baer-Hanke_boundary-conditions-for-scalar-curvature}, and we can glue the ends together to obtain a psc-metric on $M$ for which $\Sigma$ is a minimal hypersurface.
  
  \medskip

  For the final claim, we observe that the hypersurface $S\subset m\cdot M_0$ is area minimizing.
  Therefore, the psc-metric on $M_2\cup_\Sigma M_0\cup M_1$ constructed above has stably minimal and hence Yamabe-positive boundary.
  An application of \cite[Corollary B]{Akutagawa-Botvinnik_Manifolds-of-positive-scalar-curvature-and-conformal-cobordism-theory} yields a psc-metric on $M_2\cup_\Sigma M_0\cup M_1$ which is of product form near the boundary.\footnote{Note that this conformally changed metric need not agree with the original metric on $M$ away from $\Sigma$.}
  Since $M_1\colon S\leadsto \Sigma$ and $M_2^\op\colon S\leadsto \Sigma$ are $\theta$-cobordisms for $\theta$ the tangential $2$-type of $\Sigma$, we can perform surgeries on $M_1$ and $M_2^\op$ to obtain $M_1'$ and $M_2'$  for which the respective inclusions of $\Sigma$ $2$-connected and by \cref{rem:extension-of-surgery-principle} $(iii)$ there exists a psc-metric on $(M_2')^\op\cup_S M_2\cup_\Sigma M_0\cup M_1\cup_S(M_1')^\op$ which is of product form near the boundary.
  As before, $M_1\cup_S(M_1')^\op$ and $(M_2')^\op\cup_S M_2$ are both $\theta$-cobordant to cylinders over $\Sigma$, and we can hence apply \cref{prop:psc-from-doubles-to-cylinders} to obtain a psc-metric on $M_0$ which is of product form near the boundary and whose boundary restrictions can be made to agree.
  Gluing together the ends of $M_0$ yields the required psc-metric on $M$ which is of product type in a neighborhood of $\Sigma$.
\end{proof}

\begin{proof}[Proof of \cref{main:minimal-surfaces}]
  By \cref{prop:tangential-2-type-extends}, the tangential $2$-type of $\Sigma$ extends to $M$ in either case.
  Therefore, the requirements from \cref{thm:criterion-for-minimal-surface} are satisfied.
\end{proof}

\begin{remark}
  It is possible to derive an analogous result for almost spin manifolds from \cref{thm:criterion-for-minimal-surface}.
\end{remark}

\appendix
\crefalias{section}{appendix}

\section{The doubling conjecture for low-dimensional manifolds}\label{sec:low-dimensional}
The $2$-dimensional case of the doubling conjecture follows from the Gauß--Bonnet-theorem: The only two-dimensional double that admits positive scalar curvature is the $2$-sphere and the $2$-disk admits positive scalar curvature with mean convex boundary.
Furthermore, as observed in \cite[Section 7]{rosenberg-weinberger_positive-scalar-curvature-on-manifolds-with-boundary-and-their-doubles}, the $3$-dimensional case can be extracted from the following result of Carlotto--Li \cite{carlotto-li_constrained-deformations-of-positive-scalar-curvature-I}.

\begin{theorem}[{\cite[Theorem 2.1]{carlotto-li_constrained-deformations-of-positive-scalar-curvature-I}}]
  Let $M$ be a $3$-manifold with boundary $\partial M$ such that the double $\double M$ of $M$ admits positive scalar curvature.
  Then:
  \[M\cong P_{\gamma_1}\#\dots\#P_{\gamma_a}\#\faktor{S^3}{\Gamma_1}\#\dots\#\faktor{S^3}{\Gamma_b}\#\left(\underset{i=1}{\overset{c}{\#}} S^1\times S^2\right)\setminus \left(\coprod_{j=1}^d B_j^{3}\right),\]
  where $a,b,c,d\in\bbN$, $P_{\gamma_i}$ are handle bodies of genus $\gamma_i$, $\Gamma_i$ are finite subgroups of $\so(4)$ and $B_i$ are embedded balls.
\end{theorem}

Since $P_{\gamma_i}$ consist of one $0$- and some $1$-handles, \cref{lem:cobordism-principle-for-hpsc-bounding} yields a psc-metric on $P_{\gamma_i}$ with mean convex boundary.
Connected sums of psc-manifolds again admit psc-metrics by the surgery theorem for positive scalar curvature \cite[Theorem A]{gl80a}.
Hence, $P_{\gamma_1}\#\dots\#P_{\gamma_a}\#\faktor{S^3}{\Gamma_1}\#\dots\#\faktor{S^3}{\Gamma_b}\#\left(\underset{i=1}{\overset{c}{\#}} S^1\times S^2\right)$ admits positive scalar curvature with mean convex boundary.
Employing the surgery theorem again, we can assume that such a metric equals a torpedo-metric on all balls $B_j$.
Removing these balls will create additional boundary components which are of product type, hence minimal.

\section{Extension of psc-metrics onto cobordisms}
This appendix contains proofs of two extension results related to positive scalar curvature. 
The first is a classical statement on extending positive scalar curvature metrics over cobordisms due to Gajer \cite{gajer}.
The second establishes that any smooth, orientable, nullbordant manifold occurs as the mean convex boundary of an orientable manifold admitting a positive scalar curvature metric, thereby establishing the existence of positive scalar curvature fill-ins with mean convex boundary for orientable manifolds.

\begin{prop}\label{prop:extension-to-cobordism}
    Let $W\colon M_0\leadsto M_1$ be a cobordism such that $\dim(W)\ge6$ and the inclusion $M_1\embeds W$ is $2$-connected. 
    Then, any psc-metric on $M_0$ extends to a psc-metric on $W$ with cylindrical boundary.
\end{prop}

\begin{proof}
    First, we decompose $W$ into traces of surgeries and by \cref{thm:wall-geometric-connectivity}, all of these surgeries can be assumed to have codimension at least $3$.
    Therefore, it suffices to consider the case $W=\tr(\varphi)$ for an embedding $\varphi\colon S^k\times D^{n-k}\embeds M_0$ with $n-k\ge3$.
    By \cite[Theorem 1.2]{ebertfrenck}, there exists an isotopy $(g_t)_{t\in[0,1]}$ of psc-metrics on $M_0$ such that
    \begin{enumerate}
        \item $g_t$ is constant in $t$ on a neighborhood of $\{0, 1\}$. 
        \item $\varphi^*g_1 = g_\circ + g_\tor$, for $g_\circ$ the round metric and $g_\tor$ a torpedo-metric\footnote{A torpedo-metric is an $\ort(n-k)$-invariant psc-metric on $D^{n-k}$ which restricts to the round metric on the boundary}.
    \end{enumerate}
    Since isotopic psc-metrics are concordant, there exists a psc-metric $G$ on $M_0\times[0,1]$ that is cylindrical near the boundary and extends $g_i$ on $M_0\times\{i\}$ for $i=0,1$.
    Furthermore, $G$ can be extended by the psc-metric $g_\tor + g_\tor$ on $D^{k+1}\times D^{n-k}$ onto $\tr(\varphi)$ by the second property above.
\end{proof}

\begin{theorem}\label{thm:every-manifold-bounds-hpsc-metrics}
  Let $\Sigma$ be a not necessarily connected but orientable manifold of dimension $d\ge5$ that is an oriented boundary.
  Then $\Sigma$ is the mean convex boundary of an oriented manifold $M$ of positive scalar curvature.
\end{theorem}

\begin{proof}
  We start with the following observation: If $W\colon M_0\leadsto M_1$ is an oriented cobordism, we can perform surgery on the interior of $W$ to make $W$ simply connected.
  Therefore, if $M_1$ is connected, the inclusion $M_1\embeds W$ is $1$-connected and hence, $W$ consists of handles of codimension at least $2$.

  \medskip

  Let $\Sigma=\Sigma_1\amalg\dots\amalg \Sigma_n$ be the decomposition of $\Sigma$ into its components and let $X$ be a simply connected, oriented manifold with boundary $\Sigma$. 
  For $i\ge2$ we choose oriented manifolds $N_i$, which are oriented cobordant to $\Sigma_i$ and admit a psc-metric $h_i$.
  As in the proof of \cref{lem:inheritance-for-doubling-conjecture}, the metric $dt^2 + f^2h_i$ on $N_i\times [-1,1]$  for $f=1+\tfrac{t^2}{8n}$ has positive scalar curvature and strictly mean convex boundary.
  Furthermore, let $W_i$ be simply connected oriented cobordisms from $N_i$ to $\Sigma_i$.
  By \cref{lem:cobordism-principle-for-hpsc-bounding}, we can extend the metrics $dt^2 + f^2h_i$ to psc-metrics with strictly mean convex boundary on $X_i\coloneqq W_i\cup_{N_i} N_i\cup_{N_i} W_i^\op$, which is a self-cobordism of $\Sigma_i$.
  We define 
  \[M\coloneqq X\cup_{\Sigma_2\amalg\dots\amalg\Sigma_n} (X_2\amalg\dots\amalg X_n)\]
  Since $X$ is simply connected, we can extend the above psc-metrics on $X_2\amalg \dots\amalg X_n$ over $X$. 
  Hence, $M$ admits a psc-metric with mean convex boundary given by $\Sigma$.
\end{proof}

  \bigskip
  \printbibliography
  \bigskip

\end{document}